\documentclass[a4paper,11pt]{amsart}

\usepackage{mathrsfs}
\usepackage{color}
\usepackage[all]{xy}
\usepackage{here}
\usepackage{longtable}
\usepackage{multirow}
\usepackage[driverfallback=dvipdfm]{hyperref}
\usepackage{geometry}
\newtheorem{theorem}{Theorem}[section]
\newtheorem{lemma}[theorem]{Lemma}
\newtheorem{cor}[theorem]{Corollary}
\newtheorem{prop}[theorem]{Proposition}
\newtheorem{prob}[theorem]{Problem}
\newtheorem{claim}[theorem]{Claim}

\theoremstyle{definition}
\newtheorem{definition}[theorem]{Definition}
\newtheorem{eg}[theorem]{Example}

\theoremstyle{remark}
\newtheorem{remark}[theorem]{Remark}

\numberwithin{equation}{section}

\newcommand{\bA}{\mathbb{A}}
\newcommand{\bC}{\mathbb{C}}

\newcommand{\bF}{\mathbb{F}}

\newcommand{\kc}{\overline{k}}

\newcommand{\fl}{\mathfrak{l}}
\newcommand{\sM}{\mathscr{M}}

\newcommand{\sO}{\mathscr{O}}
\newcommand{\bP}{\mathbb{P}}

\newcommand{\bQ}{\mathbb{Q}}
\newcommand{\bR}{\mathbb{R}}

\newcommand{\bZ}{\mathbb{Z}}

\newcommand{\Bs}{\mathrm{Bs}}

\newcommand{\Gal}{\mathrm{Gal}(\overline{k}/k)}

\newcommand{\NE}{\overline{\mathrm{NE}}}

\newcommand{\Pic}{\mathrm{Pic}}

\newcommand{\Spec}{\mathrm{Spec}}

\begin{document}

\title{Cylinders in weak del Pezzo fibrations}

%    Remove any unused author tags.

%    author one information
\author{}
\address{}
\curraddr{}
\email{}
\thanks{}

%    author two information
\author{Masatomo Sawahara}
\address{Graduate School of Science and Engineering, Saitama University, Shimo-Okubo 255, Sakura-ku Saitama-shi,  Saitama 338-8570, JAPAN}
\curraddr{}
\email{sawahara.masatomo@gmail.com}
\thanks{This article has been accepted for publication in Transform.\ Groups (but this version is not the published version and does not reflect post-acceptance improvements, or any corrections). }

% The published version is available online at: \href{https://doi.org/10.1007/s00031-022-09730-y}{\tt https://doi.org/10.1007/s00031-022-09730-y}. }



\date{}

\dedicatory{}

\begin{abstract}
In this article, we shall look into the existence of vertical cylinders contained in a weak del Pezzo fibration as a generalization of the former work due to Dubouloz and Kishimoto in which they observed vertical cylinders found in del Pezzo fibrations. 
The essence lying in the existence of a cylinder in the generic fiber, 
we devote mainly ourselves into a geometry of minimal weak del Pezzo surfaces defined over a field of characteristic zero from the point of view of cylinders. 
As a result, we give the classification of minimal weak del Pezzo surfaces defined over a field of characteristic zero, moreover, we show that weak del Pezzo fibrations containing vertical cylinders are quite restrictive. 
\end{abstract}

\maketitle

\setcounter{tocdepth}{1}
%%%%%%%%%%%%%%%%%%%%%%%%%%%%%%%%%%%%%%%%%%%%%%%%%%%%%%%%%%%%%%%%%%%%%%%%%%%%%%%%%%%%%%%%%%%%%%%%%%%%%%%%%%%
\section{Introduction}\label{1}
%%%%%%%%%%%%%%%%%%%
Let $k$ be a field of characteristic zero. 
An open subset $U$ contained in a normal algebraic variety $X$ defined over $k$ is called an {\it $\bA ^s_k$-cylinder}, if $U$ is isomorphic to $Z \times \bA ^s_k$ for some algebraic variety $Z$. 
When the {\rm rank $s$ of cylinder} $U$ is not important, $U$ is just said to be a {\it cylinder}. 
Certainly, cylinders are geometrically simple objects, however, they receive a lot of attention recently from the viewpoint of unipotent group actions on affine cones over polarized varieties (see \cite{KPZ1, KPZ2, KPZ3, KPZ4}). 

As a special type of projective varieties, let us look at Mori Fiber Space defined over $\bC$ (MFS, for simplicity), say $f: X \to Y$. 
Let $r= \dim (X) -\dim (Y)$ be the relative dimension of $f$, where we note that the generic fiber of MFS is of Picard rank one. 
In case of $r=1$, i.e., Mori conic bundle case, a general fiber of $f$ is a smooth rational curve $\bP ^1_{\bC}$, so it contains obviously the affine line $\bA^1_{\bC}$. 
Hence to some extent, it seems reasonable to expect that a family of affine lines found in general fibers would be unified to yield an $\bA^1_{\bC}$-cylinder in $X$ respecting the structure of $f$ (in other words, a {\it vertical $\bA^1_{\bC}$-cylinder} with respect to $f$ (see Definition \ref{vertical})). 
But in fact, it follows that $X$ admits a vertical $\bA^1_{\bC}$-cylinder if and only if the generic fiber $X_\eta=f^{-1} (\eta )$ of $f$, which is isomorphic to a smooth conic in the projective plane $\bP _{\bC  (Y) }^2$ defined over the function field $\bC (Y)=\bC (\eta )$ of the base variety, has a $\bC (Y)$-rational point. 
On the other hand, as for the case of $r=2$, i.e., $f: X \to Y$ is a del Pezzo fibration, the criterion for $X$ to contain a vertical cylinder with respect to $f$ becomes to be more subtle (see \cite{DK18}), namely, $X$ admits a vertical $\bA^1_{\bC}$-cylinder if and only if the degree of the del Pezzo fibration is greater than or equal to $5$ in addition to the existence of a $\bC (Y)$-rational point on the generic fiber $X_\eta$ of $f$. 
This article will deal mainly with criteria concerning the existence of vertical cylinders found on a {\it weak del Pezzo fibration}, which is the generalization of a del Pezzo fibration (see Definition \ref{defwdp} below): 
%%%%%%%%%%%%%%%%%%%
\begin{definition}\label{defwdp} 
A dominant projective morphism $f: X \to Y$ of relative dimension two between normal varieties defined over $\bC$ such that total space $X$ has only $\bQ$-factorial terminal singularities is called a {\it weak del Pezzo fibration} if the generic fiber $X_\eta$ is a weak del Pezzo surface, which is minimal over the field $\bC (Y)$ of rational functions on the base variety (see \S \S \ref{2-1}, for definitions). 
\end{definition} 
%%%%%%%%%%%%%%%%%%%
\begin{remark}
Let $f: X \to Y$ be a weak del Pezzo fibration and let $X_{\eta}$ be the generic fiber of $f$, which is a minimal weak del Pezzo surface defined over the field $\bC (Y)$. 
Then the Picard rank $\rho (X_{\eta})$ is actually equal to either $1$ or $2$. 
Notice that $f$ is a del Pezzo fibration if and only if $\rho (X_{\eta})=1$. 
\end{remark} 
%%%%%%%%%%%%%%%%%%%
We have to define vertical cylinders which play an important role in this article: 
%%%%%%%%%%%%%%%%%%%
\begin{definition}\label{vertical}
Let $\varphi : V \to W$ be a dominant projective morphism of relative dimension $r\ge 1$ defined over $\bC$. 
An open subset $U$ of $V$ is called a {\it vertical $\bA ^s_{\bC}$-cylinder with respect to} $\varphi$ if the following two conditions hold: 
\begin{itemize} 
\item $U$ is an $\bA ^s_{\bC}$-cylinder $\bA ^s_{\bC} \times Z$ for a certain algebraic variety $Z$. 
\item There exists a dominant morphism $\psi : Z \to Y$ (of relative dimension $r-s$) such that the restriction of $\varphi$ to $U$ coincides with $\psi \circ pr_Z$. 
\end{itemize}
\end{definition}
%%%%%%%%%%%%%%%%%%%
For a weak del Pezzo fibration $f:X \to Y$ over $\bC$, by definition, provided that $X$ contains a vertical $\bA ^s_{\bC}$-cylinder with respect to $f$, the general fiber $X_y = f^{-1}(y)$ contains an $\bA ^s_{\bC}$-cylinder. 
But, the converse does not hold true in general. 
More precisely, the following fact is known: 
%%%%%%%%%%%%%%%%%%%
\begin{lemma}[{\cite[Lemma 3]{DK18}}]
Let $\varphi :V \to W$ be a dominant morphism defined over $\bC$. 
Then $\varphi$ admits a vertical $\bA ^s_{\bC}$-cylinder if and only if the generic fiber $V_{\eta} = \varphi ^{-1}(\eta )$, which is defined over the field $\bC (W)=\bC (\eta )$, contains an $\bA ^s_{\bC (W)}$-cylinder.  
\end{lemma}
%%%%%%%%%%%%%%%%%%%
The main interest in the article lies in a criterion about the existence of a vertical cylinder found in weak del Pezzo fibrations $f: X \to Y$ over $\bC$. 
As just above mentioned, $X$ contains a vertical $\bA ^s_{\bC}$-cylinder with respect to $f$ if and only if the generic fiber $X_\eta$ of $f$, which is a minimal weak del Pezzo surface defined over the field $\bC (\eta )=\bC (Y)$, contains an $\bA ^s _{\bC (Y)}$-cylinder. 
Thus, the following problem is essential for our purpose: 
%%%%%%%%%%%%%%%%
\begin{prob}\label{prob1}
Let $k$ be a field of characteristic zero, and let $S$ be a minimal weak del Pezzo surface defined over $k$. Then: 
\begin{enumerate} 
\item Classify minimal weak del Pezzo surfaces defined over $k$. 
\item In which case does $S$ contain an $\bA ^1_k$-cylinder, or more idealistically the affine plane $\bA ^2_k$? 
\end{enumerate} 
\end{prob} 
%%%%%%%%%%%%%%%%%%%
The main results in the article, which is concerned with Problem \ref{prob1}, are summarized in the following two theorems. 

As for Problem \ref{prob1}(1), it is known that any minimal weak del Pezzo surface of degree $4$ over an arbitrary perfect field with anti-canonical divisor not ample is an Iskovskikh surface (see {\cite[Theorem 7.2]{CT88}}). 
Recently, Tamanoi studied minimal weak del Pezzo surfaces of degree $2$ over an arbitrary perfect field with anti-canonical divisor not ample ({\cite{T}}). 
The first result, which completely includes these results for the case of characteristic zero, is summarized as follows: 
%%%%%%%%%%%%%%%%%%%
\begin{theorem}\label{main(1)}
Let $k$ be a field of characteristic zero and let $S$ be a weak del Pezzo surface, whose $-K_S$ is not ample, defined over $k$ of degree $d:=(-K_S^2)$. 
Then $S$ is minimal if and only if $\rho (S)=2$ and the type of $S$ is one of the following (for the definition of the type of $S$, see \S \S \ref{2-3}): 
\begin{itemize}
\item $d=8$ and $A_1$-type. 
\item $d=4$ and $(2A_1)_<$-type. 
\item $d=2$ and $A_1$, $A_2$ or $(4A_1)_>$-type. 
\item $d=1$ and $2A_1$ or $2A_2$-type. 
\end{itemize}
\end{theorem}
%%%%%%%%%%%%%%%%%%%
On the other hand, it is known that the degree of a minimal del Pezzo surface (i.e., a minimal weak del Pezzo surface, whose anti-canonical divisor is ample) $S$ with $\rho (S)=2$ defined over a field $k$ of characteristic zero is equal to $1$, $2$, $4$ or $8$ (see Appendix \ref{5}). 
Thus, minimal weak del Pezzo surfaces are already somehow restrictive objects. 

The next result, which will yield a complete answer to Problem \ref{prob1}(2), says that those containing cylinders are a quite minority: 
%%%%%%%%%%%%%%%%%%%
\begin{theorem}\label{main(2)}
Let $k$ be a field of characteristic zero, and let $S$ be a minimal weak del Pezzo surface of degree $d:=(-K_S^2)$ and with $\rho (S)=2$ defined over $k$. Then: 
\begin{enumerate}
\item $S$ contains an $\bA ^1_k$-cylinder if and only if $d=8$ and $S$ is endowed with a structure of Mori conic bundle admitting a section defined over $k$. 
\item $S$ contains the affine plane $\bA ^2_k$ if and only if $d=8$ and $S(k) \not= \emptyset$. 
\end{enumerate}
\end{theorem}
%%%%%%%%%%%%%%%%%%%
Note that Theorem \ref{main(2)} excludes the case of Picard rank one since it has already been studied by {\cite{DK18}}. 
As an application of Theorems \ref{main(1)} and \ref{main(2)}, we obtain the following result: 
%%%%%%%%%%%%%%%%%%%
\begin{cor}\label{cor1}
Let $f:X \to Y$ be a weak del Pezzo fibration but not a del Pezzo fibration defined over $\bC$, let $X_{\eta}$ be the generic fiber of $f$ and let $d$ be the degree of $f$, i.e., the degree of $X_{\eta}$. Then: 
\begin{enumerate}
\item $d\in \{ 1,2,4,8\}$. 
\item $f$ admits a vertical $\bA ^1_{\bC}$-cylinder if and only if $d=8$ and $X_{\eta}$ is endowed with a structure of Mori conic bundle admitting a section defined over $\bC (\eta )=\bC (Y)$. 
\item $f$ admits a vertical $\bA ^2_{\bC}$-cylinder if and only if $d=8$ and $X_{\eta}(\bC (Y)) \not= \emptyset$. 
\end{enumerate}
\end{cor}
%%%%%%%%%%%%%%%%%%%
Furthermore, we shall consider a weak del Pezzo fibration $f:X \to Y$ of degree $8$ such that $Y$ is a curve, so that $X$ is a threefold. 
By using the Tsen's theorem to the generic fiber $X_{\eta}$ of $f$, we see that $X_{\eta}$ allows a $\bC (Y)$-rational point (see also {\cite[Theorem 3.12]{H09}}). 
Thus, by Corollary \ref{cor1}(3) and {\cite[Corollary 14]{DK18}}, we finally obtain the following corollary: 
%%%%%%%%%%%%%%%%%%%
\begin{cor}
Let $f:X \to Y$ be a weak del Pezzo fibration defined over $\bC$. 
If $f$ is of degree $8$ and $Y$ is a curve, then $f$ always admits a vertical $\bA ^2_{\bC}$-cylinder. 
\end{cor}
%%%%%%%%%%%%%%%%%%%
The scheme of the article proceeds as follows: 
%%%%%%%%%%%%%%%%%%%
In $\S \ref{2}$, we shall summarize basic properties on weak del Pezzo surfaces $S$ defined over a field $k$ of characteristic zero. It is well known that the degree $d=(-K_S^2)$ of $S$ is in the range $1 \leq d \leq 8$. 
%%%%%%%%%%%%%%%%%%%
In $\S \ref{3}$, we shall give the proof of Theorem \ref{main(1)}. 
For this purpose, we need to calculate intersection numbers related to certain divisors on weak del Pezzo surfaces defined over an algebraically closed field in order to look for some specific $(-1)$-curves therein. 
%%%%%%%%%%%%%%%%%%%
In $\S \ref{4}$, we shall give the proof of Theorem \ref{main(2)}. 
This proof will be divided according to the degree $d$, more precisely the case of $d=8$ and the case of $d<8$, separately. 
In $\S \S \ref{4-1}$, at first we deal with the case of $d=8$, i.e., the case where $S$ is a $k$-form of $\bP ^1_{\kc} \times \bP ^1_{\kc}$ or the Hirzebruch surface $\bF _2$ of degree two. 
We will notice however that $S$ does not necessarily admit $k$-rational points (compare the fact that any del Pezzo surface of Picard rank $1$ over $k$ containing a cylinder admits $k$-rational points, see \cite{DK18}). 
In $\S \S \ref{4-2}$,  we deal with the case of $d<8$, so that $d$ is equal to either $1$, $2$ or $4$ by the result in \S \ref{3}, and prove that $S$ does not contain any cylinder in consideration of the result in $\S \ref{2}$. One of the most important ingredients for the proof at this step is a variant of Corti's inequality (cf. \cite[Theorem 3.1]{C00}). 
%%%%%%%%%%%%%%%%%%%
In Appendix \ref{5}, we will give the proof for the fact that the degree of a minimal smooth del Pezzo surface of Picard rank $2$ is equal to $1,2,4$ or $8$. 
Although this fact seems to be well known, we will yield the proof for the readers' convenience since we could not find proof in the literature. 
%%%%%%%%%%%%%%%%%%%
Appendix \ref{6} will summarize the classification of weak del Pezzo surfaces, whose anti-canonical divisor is not ample, defined over an algebraically closed field of characteristic zero. 
%%%%%%%%%%%%%%%%%%%%%%%%%%%%%%%%%%%%%%%%%%%%%%%%%%%%%%%%%%%%%%%%%%%%%%%%%%%%%%%%%%%%%%%%%%%%%%%%%%%%%%%%%%%
\section{Preliminaries}\label{2}
%%%%%%%%%%%%%%%%%%%
In this section, let $k$ be a field of characteristic zero. 
%%%%%%%%%%%%%%%%%%%%%%%%%%%%%%%%%%%%%%%%%%%%%%%%%%%%%%%%%%%%%%%%%%%
\subsection{Basic properties about weak del Pezzo surfaces}\label{2-1}
In this article, a {\it weak del Pezzo surface} means a smooth projective surface such that its anti-canonical divisor is nef and big. 
In this subsection, we prepare the basic but important properties of weak del Pezzo surfaces in the subsequent argument. 

Let $S$ be a weak del Pezzo surface defined over $k$ and let $S_{\kc}$ be the base extension of $S$ to algebraic closuer $\kc$, i.e., $S_{\kc} := S \times _{\Spec (k)} \Spec (\kc )$. 
Then the following two lemmas hold: 
%%%%%%%%%%%%%%%%%%%
\begin{lemma}\label{bir}
With the notation as above, then $S_{\kc}$ is isomorphic to $\bP ^1_{\kc} \times \bP ^1_{\kc}$ or the Hirzebruch surface $\bF _2$ of degree $2$, or a blow-up at most eight points, which may include infinitely near points, from $\bP ^2 _{\kc}$. 
\end{lemma}
%%%%%%%%%%%%%%%%%%%
\begin{proof}
See, e.g., {\cite[Theorem 8.1.15]{D12}}. 
\end{proof}
%%%%%%%%%%%%%%%%%%%
\begin{lemma}\label{(-1)-curve}
With the notation as above, let $D$ be a divisor on $S_{\kc}$. 
If $(D ^2)=-1$, $(D \cdot -K_S)=1$ and $(D \cdot M) \ge 0$ for any $(-2)$-curve $M$ on $S_{\kc}$, then $D$ is a $(-1)$-curve. 
\end{lemma}
%%%%%%%%%%%%%%%%%%%
\begin{proof}
See {\cite[Lemma 8.2.22]{D12}}. 
\end{proof}
%%%%%%%%%%%%%%%%%%%
In this article, we shall mainly deal with minimal weak del Pezzo surfaces, where a weak del Pezzo surface $S$ defined over $k$ is {\it minimal} over $k$ if any orbit of a $(-1)$-curve $E$ on the base extension $S_{\kc}$ with respect to the Galois action $\Gal$ properly contains $E$ and is not a disjoint union of $(-1)$-curves. 
The following property of minimal weak del Pezzo surfaces is known: 
%%%%%%%%%%%%%%%%%%%
\begin{prop}\label{conic bdl}
With the notation as above, assume further that $S$ is minimal over $k$ and of the Picard rank $\rho (S)$ strictly more than $1$. 
Then $\rho (S)$ is equal to $2$ and $S$ is endowed with a structure of Mori conic bundle defined over $k$, i.e., it is equipped with a morphism $\pi: S \to B$ over a smooth projective curve $B$ defined over $k$ such that any fiber of the base extension $\pi_{\kc}: S_{\kc} \to B_{\kc}$ is isomorphic to the plane conic (not necessarily irreducible). 
\end{prop}
%%%%%%%%%%%%%%%%%%%
\begin{proof}
Since $S$ is minimal and the canonical divisor $K_S$ of $S$ is not nef by assumption, we obtain the assertion by {\cite[Theorem 9.3.20]{P17}}. 
\end{proof}
%%%%%%%%%%%%%%%%%%%%%%%%%%%%%%%%%%%%%%%%%%%%%%%%%%%%%%%%%%%%%%%%%%%
\subsection{Mori conic bundle on minimal weak del Pezzo surfaces}\label{2-2}
In this article, we say that a surjective morphsim $\pi :X \to Y$ between algebraic varieties defined over $k$ is a {\it $\bP ^1$-fibration} (resp. {\it $\bP ^1$-bundle}) if a general fiber (resp. any fiber) of the base extension $\pi_{\kc}: X_{\kc} \to Y_{\kc}$ is isomorphic to $\bP ^1_{\kc}$. 
Notice that the above definitions for the case that $k$ is algebraically closed coincide with {\cite[Definition 12.4]{M}}. 
Let $S$ be a weak del Pezzo surface defined over $k$ of degree $d$ and of Picard rank $\rho (S) >1$, which is minimal over $k$. 
By Proposition \ref{conic bdl}, it then follows that $\rho (S)=2$ and $S$ is endowed with a structure of Mori conic bundle defined over $k$. 
In this subsection, we shall prepare the basic properties of this Mori conic bundle for later use. 
%%%%%%%%%%%%%%%%%%%
\begin{lemma}\label{bdl}
With the notation as above, let $\pi : S \to B$ be a Mori conic bundle over $k$. Then: 
\begin{enumerate}
\item $B_{\kc} \simeq \bP ^1_{\kc}$. 
\item $\pi _{\kc}: S_{\kc} \to B_{\kc}$ is a $\bP ^1$-bundle if and only if $d = 8$. 
\item Assume that $d<8$. Then $\pi$ does not admit a section defined over $k$. 
\end{enumerate}
\end{lemma}
%%%%%%%%%%%%%%%%%%%
\begin{proof}
In (1) and (2), see {\cite[Exercise 3.13]{KSC04}}. 
We shall show (3). 
By (1), we have $B_{\kc} \simeq \bP ^1_{\kc}$. 
Note that the base extension of $\pi$ to the algebraic closure $\pi _{\kc} : S_{\kc} \to B_{\kc} \simeq \bP ^1_{\kc}$ always admits a section defined over $\kc$, by the Tsen's theorem. 
Let $\Gamma$ be a section of $\pi _{\kc}$. 
By the assumption that $d<8$ and (2), $\pi_{\kc}$ admits a singular fiber $F$. 
We can easily see by the minimality of $S$ that $F$ is the union $E + E'$ of $(-1)$-curves $E$ and $E'$ on $S_{\kc}$ meeting transversely at a point, say $p$, in such a way that $E$ and $E'$ lie in the same $\Gal$-orbit. 
Since $\Gamma$ is a section of $\pi _{\kc}$, $\Gamma$ does not pass through $p$. 
Hence, we may assume that there exists a closed point $q \in E \backslash \{ p\}$ such that $\Gamma$ passes through $q$. 
Since $E$ and $E'$ lie in the same $\Gal$-orbit, there exists a closed point $q' \in E' \backslash \{ p\}$ such that $q$ and $q'$ are contained in the same $\Gal$-orbit. 
This implies that $\Gamma$ is not defined over $k$. 
\end{proof}
%%%%%%%%%%%%%%%%%%%
The following two lemmas will play important roles in \S \S \ref{4-2}: 
%%%%%%%%%%%%%%%%%%%
\begin{lemma}\label{MCB}
With the notation as above, any $\bP ^1$-fibration $\pi : S \to B$ over a geometrically rational curve $B$ defined over $k$ is a Mori conic bundle. 
\end{lemma}
%%%%%%%%%%%%%%%%%%%
\begin{proof}
Assume that $\pi _{\kc}$ admits a singular fiber $F$. 
Since $S$ is minimal over $k$, we know that $F$ does not contain any $(-2)$-curve by {\cite[Lemma 1.5]{Z88}}. 
Moreover, $F$ is the union $E_1 + E_2$ of two $(-1)$-curves $E_1$ and $E_2$ on $S_{\kc}$ meeting transversely at a point in such a way that $E_1$ and $E_2$ lie in the same $\Gal$-orbit. 
This implies that $\pi$ is a Mori conic bundle. 
\end{proof}
%%%%%%%%%%%%%%%%%%%
\begin{lemma}\label{MCBs}
With the notation as above, assume further that $S(k) \not= \emptyset$, $-K_S$ is ample, and $d$ is equal to $1$, $2$ or $4$. 
Then $S$ is endowed with two distinct structures of Mori conic bundles $\pi _i :S \to \bP ^1_k$ defined over $k$ for $i=1,2$ such that $F_1+F_2 \sim \frac{4}{d}(-K_S)$, where $F_i$ is a general fiber of $\pi _{i,\kc}$, which is defined over $k$, for $i=1,2$. 
\end{lemma}
%%%%%%%%%%%%%%%%%%%
\begin{proof}
For any Mori conic bundle $\pi : S \to B$ over $k$, note that $B \simeq \bP ^1_k$, in particular, there exists a general fiber of $\pi _{\kc}$, which is defined over $k$. 
Indeed, since $S$ contains a $k$-rational point, so it its image via $\pi$ (see also Lemma \ref{bdl}(1) and {\cite[Proposition 4.5.10]{P17}}). 

By Proposition \ref{conic bdl}, we see that $\rho (S)=2$ and $S$ is endowed with a structure of Mori conic bundle $\pi _1 :S \to \bP ^1_k$ defined over $k$. 
In particular, there exists a general fiber $F_1$ of $\pi _1$, which is geometrically irreducible. 
By $\rho (S) =2$, the Mori cone $\NE (S)$ contains exactly two extremal rays, say $R_1$ and $R_2$ (cf.\ {\cite[\S 1.3]{KM98}}). 
Moreover, we can assume $R_1 = \bR _{\ge 0}[F_1]$ and we write $R_2 = \bR _{\ge 0}[\ell ]$ for some curve $\ell$ on $S$. 
Noticing that $\frac{4}{d}$ is an integer by $d \in \{ 1,2,4\}$, let $D$ be the divisor on $S$ defined by $D := \frac{4}{d}(-K_S) - F_1$. 
By the Riemann-Roch theorem combined with $(D^2) =0$ and $(-K_S \cdot D) =2$, we obtain $\dim |D_{\kc}| \ge 1$, in particular, $D_{\kc}$ is linearly equivalent to a union $\sum _{i=1}^rC_i$ of some irreducible curves $\{ C_i\}_{1 \le i \le r}$ on $S_{\kc}$. 
Since $-K_S$ is ample, we have $r \le 2$ by $(-K_S \cdot D) =2$, moreover, we can easily see that there are at most finitely many unions $C_1+C_2$ of two irreducible curves $C_1,C_2$ on $S_{\kc}$ with $C_1+C_2 \sim D_{\kc}$. 
Hence, there exists an irreducible curve $\Gamma$ on $S_{\kc}$ such that $D_{\kc} \sim \Gamma$. 
Let $\Gamma '$ be a $\Gal$-orbit of $\Gamma$. 
Thus, we can write $[\Gamma'] =a_1[F_1]+a_2[\ell ]$ in $\NE (S)$ for some non-negative real numbers $a_1,a_2$. 
By $({\Gamma '}^2) =0$ and $(F_1 \cdot \Gamma ') >0$, we obtain $a_1=0$. 
Namely, $\Gamma ' \in R_2$. 
This implies that there exists a Mori conic bundle $\pi _2:S \to \bP ^1_k$, which is different from $\pi _1$, such that a general fiber of $\pi _{2,\kc}$ is linearly equivalent to $\Gamma$ on $S_{\kc}$. 
Furthermore, there exists a general fiber $F_2$ of $\pi _{2,\kc}$, which is defined over $k$. 
By construction of $\pi _2$, we know $F_1+F_2 \sim \frac{4}{d}(-K_S)$. 
\end{proof}
%%%%%%%%%%%%%%%%%%%
\begin{remark}\label{MCBs(2)}
Assuming that $-K_S$ is not ample, then we have either $(\ell _1^2) \not= 0$ or $(\ell _2^2) \not= 0$ for two curves $\ell _1$ and $\ell _2$ on $S$ such that $\NE (S) = \bR _{\ge 0}[\ell _1] + \bR _{\ge 0}[\ell _2]$. 
Otherwise, we obtain $(\ell _1 \cdot \ell _2) >0$ by virtue of $(-K_S^2)>0$, however, this contradicts $(-K_S \cdot M)=0$, where $M$ is a $\Gal$-orbit of a $(-2)$-curve on $S_{\kc}$. 
Hence, the assertion of Lemma \ref{MCBs} is not true unless $-K_S$ is ample. 
y $(-K_S^2)>0$, however, it is a contradiction to $(-K_S \cdot M)=0$. 
\end{remark}
%%%%%%%%%%%%%%%%%%%%%%%%%%%%%%%%%%%%%%%%%%%%%%%%%%%%%%%%%%%%%%%%%%%
\subsection{Types of weak del Pezzo surfaces}\label{2-3}
We recall a classification of weak del Pezzo surfaces over an algebraically closed field of characteristic zero, moreover, we define the {\it type} of weak del Pezzo surfaces. 
Almost all parts of this subsection depend on {\cite{AN06, D14, CPW16}} and {\cite[\S 8]{D12}}  (see also {\cite{CP21, CT88, D34, MS}}). 
Let $S$ be a weak del Pezzo surface defined over $k$ of degree $d$ such that $-K_S$ is not ample, and let $S_{\kc}$ be the base extension of $S$. 
%%%%%%%%%%%%%%%%%%%
If $d=8$, then $S_{\kc}$ is the Hirzebruch surface $\bF _2$ of degree $2$. 
Namely, we have the contraction $\tau :S_{\kc} \to \bP (1,1,2)$ of the minimal section. 
In what follows, we shall consider the case of $d \le 7$. 
We prepare the following definition: 
%%%%%%%%%%%%%%%%%%%
\begin{definition}\label{same type}
Letting $S_1$ and $S_2$ be two weak del Pezzo surfaces over $\kc$, we say that these surfaces have the {\it same type} if there is an isomorphism $\Pic (S_1) \simeq \Pic (S_2)$ preserving the intersection form that gives a bijection between their sets of classes of $(-1)$-curves and $(-2)$-curves. 
\end{definition}
%%%%%%%%%%%%%%%%%%%
A classification of weak del Pezzo surfaces up to type seems well known to the experts, but for the reader's convenience, we recall it in what follows. Since there exists a birational morphism $\sigma :S_{\kc} \to \bP ^2_{\kc}$ by Lemma \ref{bir} and the assumption $d \le 7$, we can write $\Pic (S_{\kc}) \simeq \bigoplus _{i=0}^{9-d}\bZ e_i$ preserving the intersection form such that $(e_0^2)=1$, $(e_i^2)=-1$ for $i>0$ and $(e_i \cdot e_j)=0$ for $i,j \ge 0$ with $i \not= j$ (see \ref{3-2-1}, for details). 
Let $R(S_{\kc})$ be the subset of $\Pic (S_{\kc})$ preserving the intersection form defined by: 
%%%%%%%%%%%%%%%%%%%
\begin{align*}
R(S_{\kc}) := \{ D \in \Pic (S_{\kc})\, |\, (D^2)=-2,\ (D \cdot -K_{S_{\kc}}) = 0\} . 
\end{align*}
%%%%%%%%%%%%%%%%%%%
By {\cite[Lemma 8.2.6, Proposition 8.2.7]{D12}}, $R(S_{\kc})$ is the root system of type $A_1$, $A_2+A_1$, $A_4$, $D_5$ and $E_{9-d}$ if $d=7$, $d=6$, $d=5$, $d=4$ and $d \le 3$, respectively. 
By {\cite[Proposition 8.2.25]{D12}}, the number $r$ of all $(-2)$-curves on $S_{\kc}$ is less than $10-d$, moreover, letting $M_1,\dots ,M_r$ be all $(-2)$-curves on $S_{\kc}$, the sublattice $L(S_{\kc})$, which is generated by $M_1,\dots ,M_r$, in $R(S_{\kc})$ is a root lattice of rank $r$ corresponding to the intersection matrix with respect to these $(-2)$-curves. 
That is, $L(S_{\kc})$ determines a subsystem of the root system $R(S_{\kc})$. 
More precisely, noticing that the field $k$ is of characteristic zero, $L(S_{\kc})$ is one of the following according to the degree $d$: 
%%%%%%%%%%%%%%%%%%%
\begin{itemize}
\item $d=7$: the root system of type $A_1$; 
\item $d=6$ (resp. $d=5$, $d=4$, $d=3$): the subsystem of the root system of type $A_2+A_1$ (resp. $A_4$, $D_5$, $E_6$); 
\item $d=2$: the subsystem of the root system of type $E_7$ except for type of $7A_1$; 
\item $d=1$: the subsystem of the root system of type $E_8$ except for types of $7A_1$, $8A_1$ and $D_4+4A_1$. 
\end{itemize}
%%%%%%%%%%%%%%%%%%%
Noting that the type of root system $L(S_{\kc})$ corresponds to the dual graph of $\sum _{i=1}^rM_i$, we obtain the contraction $\tau :S_{\kc} \to \widetilde{S}$ of all $(-2)$-curves on $S_{\kc}$, where $\widetilde{S}$ is a normal singular del Pezzo surface over $\kc$ with at most Du Val singularities, say {\it Du Val del Pezzo surface} over $\kc$ for short in this article. 
Conversely, for any Du Val del Pezzo surface over $\kc$, its minimal resolution is a weak del Pezzo surface over $\kc$. 
Hence, types of singularities of Du Val del Pezzo surfaces have a one-to-one correspondence with types of root systems of their minimal resolution. 
On the other hand, notice that $\tau$ is defined over $k$ by the construction of $\tau$. This fact will be used in Lemma \ref{min(1)}. 

Now, we say that the type of singularities of $\widetilde{S}$ is called ``Singularities" of $S$. 
Furthermore, we say that the number of $(-1)$-curves on $S_{\kc}$ is called ``$\#$\,Lines" of $S$, where ``$\#$\,Lines" is finite by Lemma \ref{(-1)-curve} and {\cite[Proposition 8.2.19]{D12}}. 
In this article, the triplet $(d,\, \text{Singularities},\, \# \, \text{Lines})$ is called the {\it type} of $S$. 
For two weak del Pezzo surfaces $S_1$ and $S_2$ over $k$, it is known that the types of $S_1$ and $S_2$ (in the sense of the above triplet) are the same if and only if $S_{1,\kc}$ and $S_{2,\kc}$ have the same type (in the sense of Definition \ref{same type}). 
Moreover, it is known that all pairs $(d,\, \text{Singularities})$ can uniquely determine the number of ``$\#$\,Lines" except for the following pairs: 
%%%%%%%%%%%%%%%%%%%
\begin{align}\label{distinct type}
\begin{aligned}
(d,\, \text{Singularities}) = &(6,A_1),\, (4,A_3),\, (4,2A_1), \\
&(2, A_5+A_1),\, (2, A_5),\, (2, A_3+2A_1),\, (2,A_3+A_1),\, (2,4A_1),\, (2, 3A_1),\\ 
&(1, A_7),\, (1, A_5+A_1),\, (1, 2A_3),\, (1,A_3+2A_1),\, (1, 4A_1). 
\end{aligned}
\end{align}
%%%%%%%%%%%%%%%%%%%
On the other hand, if the pair $(d,\, \text{Singularities})$ is one of those in the list of (\ref{distinct type}), then it is known that there are exactly two possibilities of the number of ``$\#$\,Lines". 

To simplify the notation, we introduce the notation for the type of weak del Pezzo surfaces instead of the above triplet as follows: 
Let $S$ be a weak del Pezzo surface over $k$ such that the pair $(d, X)$ of the degree and ``Singularities" of $S$ is not in the list in (\ref{distinct type}). 
Then we say that $S$ is of $X$-type. 
On the other hand, let $S_1$ and $S_2$ be two weak del Pezzo surfaces over $k$ such that pairs of the degree and ``Singularities" of them are the same, and their common pair $(d, X)$ is one of those in the list of (\ref{distinct type}). 
Moreover, assume that $\#$\,Lines of $S_1$ is strictly more than $\#$\,Lines of $S_2$. 
Then we say that $S_1$ (resp. $S_2$) is of $(X)_>$-type (resp. $(X)_<$-type). 
The detail is summarized in Table \ref{type} in Appendix \ref{6}, for the reader's convenience. 

The following two cases will play an important role in \S \ref{3}: 
%%%%%%%%%%%%%%%%%%%
\begin{eg}
Let $S$ be a weak del Pezzo surface of degree $d$ over a field of characteristic zero. 
Let us look at cases $(d,\, {\rm Singularities})=(4,2A_1),\, (2,4A_1)$. There are two possibilities about $\#$\,Lines for each of such cases as follows: 
\begin{itemize}
\item In case of $(d,\, {\rm Singularities})=(4,2A_1)$, if $S$ is of $(2A_1)_>$-type (resp. $(2A_1)_<$-type), then $\# \, \text{Lines}=9$ (resp. $\# \, \text{Lines}=8$). 
\item In case of $(d,\, {\rm Singularities})=(2,4A_1)$, if $S$ is of $(4A_1)_>$-type (resp. $(4A_1)_<$-type), then $\# \, \text{Lines}=20$ (resp. $\# \, \text{Lines}=19$). 
\end{itemize}
As for how to calculate $\#$\,Lines, see Example \ref{inB} in Appendix \ref{6}. 
\end{eg}
%%%%%%%%%%%%%%%%%%%%%%%%%%%%%%%%%%%%%%%%%%%%%%%%%%%%%%%%%%%%%%%%%%%%%%%%%%%%%%%%%%%%%%%%%%%%%%%%%%%%%%%%%%%
\section{Proof of Theorem \ref{main(1)}}\label{3}
%%%%%%%%%%%%%%%%%%%
In this section, we will prove Theorem \ref{main(1)}. 
Let $k$ be a field of characteristic zero and let $S$ be a weak del Pezzo surface, whose $-K_S$ is not ample, of degree $d$ over $k$. 
%%%%%%%%%%%%%%%%%%%%%%%%%%%%%%%%%%%%%%%%%%%%%%%%%%%%%%%%%%%%%%%%%%%
\subsection{Quasi-minimal weak del Pezzo surfaces}\label{3-1}
The purpose of this section is to classify minimal weak del Pezzo surfaces with anti-canonical divisor not ample. 
In order to state this classification, we shall introduce a weaker version of being minimal, which depends only on degree and type, the so-called being {\it quasi-minimal}. 
%%%%%%%%%%%%%%%%%%%
\begin{lemma}\label{min(1)}
With the notation as above, assume further that $\rho (S)=2$. 
Then the type of $S$ is either $mA_1$-type or $mA_2$-type for some positive integer $m$. 
In particular, the type of $S$ is one of the following: 
\begin{itemize}
\item $d=7$ or $8$ and $A_1$-type. 
\item $d=6$ and $A_2$, $2A_1$, $(A_1)_<$ or $(A_1)_>$-type. 
\item $d=5$ and $A_2$, $2A_1$ or $A_1$-type. 
\item $d=4$ and $4A_1$, $3A_1$, $A_2$, $(2A_1)_<$, $(2A_1)_>$ or $A_1$-type. 
\item $d=3$ and $3A_2$, $2A_2$, $4A_1$, $3A_1$, $A_2$, $2A_1$ or $A_1$-type. 
\item $d=2$ and $3A_2$, $6A_1$, $5A_1$, $2A_2$, $(4A_1)_<$, $(4A_1)_>$, $(3A_1)_<$, $(3A_1)_>$, $A_2$, $2A_1$ or $A_1$-type. 
\item $d=1$ and $4A_2$, $3A_2$, $6A_1$, $5A_1$, $2A_2$, $(4A_1)_<$, $(4A_1)_>$, $3A_1$, $A_2$, $2A_1$ or $A_1$-type. 
\end{itemize}
\end{lemma}
%%%%%%%%%%%%%%%%%%%
\begin{proof}
At first, we show that the type of $S$ is either $mA_1$-type or $mA_2$-type for some positive integer $m$. 
Let $\tau :S \to \widetilde{S}$ be the contraction of all $(-2)$-curves on $S_{\kc}$, where $\tau$ is defined over $k$ (see \S \S \ref{2-3}). 
By virtue of $1 \le \rho (\widetilde{S}) < \rho (S) =2$, it follows that $\widetilde{S}$ is a Du Val del Pezzo surface of $\rho (\widetilde{S}) =1$. 
Hence, we obtain $\rho (S) - \rho (\widetilde{S}) =1$. 
This implies that all $(-2)$-curves on $S_{\kc}$ lie in the same $\Gal$-orbit. 
Thus, it must be that $S$ is either of $mA_1$ or $mA_2$ for some positive integer $m$ and all singularities on $\widetilde{S}_{\kc}$ are transformed to each other by means of the action of $\Gal$. 
Otherwise, by the dual graph of the union of all $(-2)$-curves on $S_{\kc}$, we can easily see $\rho (S) - \rho (\widetilde{S}) >1$, which is a contradiction. 
Moreover, the remaining assertion follows from the above argument by combined with the classification of weak del Pezzo surfaces over algebraically closed fields of characteristic zero (see Table \ref{type} in Appendix \ref{6}). 
\end{proof}
%%%%%%%%%%%%%%%%%%%
If $S$ is minimal, then the type of $S$ is one of those in the list of Lemma \ref{min(1)} by Proposition \ref{conic bdl}. 
%%%%%%%%%%%%%%%%%%%
\begin{eg}\label{ex1}
We shall construct an example of a weak del Pezzo surface of degree $4$, which is minimal over $\bR$ as follows: 
Let $\widetilde{S}'$ be the cubic surface defined by: 
\begin{align*}
\widetilde{S}' := \left( \, xy (x + y) - (x - y) (z^2 + w^2) = 0\, \right) \subseteq \bP ^3_{\bR} = {\rm Proj}(\bR [x,y,z,w]). 
\end{align*}
$\widetilde{S}'_{\bC}$ has two singular points $p _{\pm} := [0\! :\! 0\! :\!1\! :\! \pm \sqrt{-1}] \in \widetilde{S}'_{\bC}$, which are Du Val singular points of type $A_2$. 
Let $\tau : S' \to \widetilde{S}'$ be the blow-up at $p_+ + p_-$ over $\bR$, where $p_+ + p_-$ is defined over $\bR$ so that $S'$ is a weak del Pezzo surface of degree $3$ over $\bR$ and of $2A_2$-type. 
In particular, $S' _{\bC}$ contains seven $(-1)$-curves. 
Let $E_0'$, $E_{1,\pm}'$, $E_{2,\pm}'$ and $E_{3\pm}'$ be these $(-1)$-curves on $S'_{\bC}$, which are the strict transform by $\tau _{\bC}$ of curves $\left( \, x=y=0\, \right)$, $\left( \, x=0,\ z = \pm \sqrt{-1} w\, \right)$, $\left( \, y=0,\ z = \pm \sqrt{-1} w\, \right)$ and $\left( \, x+y=0,\ z = \pm \sqrt{-1} w\, \right)$ on $\bP ^3_{\bC}$, respectively. 
Thus, $E'_0$ is defined over $\bR$, on the other hand, $E'_{i,+}$ and $E'_{i,-}$ are contained in the same ${\rm Gal} (\bC /\bR )$-orbit for $i=1,2,3$. 
Hence, we can contract $E'_0$ over $\bR$, say $\rho : S' \to S$ of $E_0'$. By construction, $S$ is then a weak del Pezzo surface of degree $4$ over $\bR$ and of $(2A_1)_<$-type, in other words, $S$ is a minimal resolution of an Iskovskikh surface (see, e.g., {\cite[\S 7]{CT88}}). 
By construction, $S$ is clearly minimal over $\bR$. 
\end{eg}
%%%%%%%%%%%%%%%%%%%
\begin{remark}\label{ex2}
Note that the minimality of weak del Pezzo surfaces can not be detected by the type only. 
For instance, if we change the defining equation of $\widetilde{S}'$ in Example \ref{ex1} to $xy(x+y)-(x-y)(z^2-w^2) \in \bR [x,y,z,w]$, then $S$ is also a weak del Pezzo surface of degree $4$ and of $(2A_1)_<$-type but not $\bR$-minimal. 
\end{remark}
%%%%%%%%%%%%%%%%%%%
Now, letting $E$ be any $(-1)$-curve on $S_{\kc}$, if $S_{\kc}$ is minimal, then there exists a $(-1)$-curve $E'$ on $S_{\kc}$ such that $(E \cdot E')>0$ and $|\sM_E (i,j) | = |\sM_{E'}(i,j)|$ for $i=1,2$ and $j=1,2$, where $\sM_C (i,j)$ is the set defined by: 
\begin{align*}
\sM_C (i,j) := \{ M\, |\, M:(-i)\text{-curve on } S_{\kc},\ (C \cdot M) = j\}
\end{align*}
for $i=1,2$, $j=1,2$ and a projective curve $C$ on $S_{\kc}$. 
By noticing this observation, we shall define a weaker version of minimality as follows: 
%%%%%%%%%%%%%%%%%%%
\begin{definition}\label{qm}
Let the notation be the same as above. 
Then $S$ is {\it quasi-minimal} if the following two conditions hold: 
\begin{itemize}
\item $S$ is either of $mA_1$-type or $mA_2$-type for some positive integer $m$. 
\item For any $(-1)$-curve $E$ on $S_{\kc}$, there exists a $(-1)$-curve $E'$ on $S_{\kc}$ such that $(E \cdot E')>0$ and $|\sM_E (i,j) | = |\sM_{E'}(i,j)|$ for $i=1,2$ and $j=1,2$. 
\end{itemize}
\end{definition}
%%%%%%%%%%%%%%%%%%%
By definition, if $S$ is minimal, then $S$ is quasi-minimal. 
Furthermore, we actually see that quasi-minimality depends only on the type by the classification of weak del Pezzo surfaces over algebraically closed fields of characteristic zero (see also Definition \ref{same type}). 

Theorem \ref{main(1)} is a consequence of the following proposition: 
%%%%%%%%%%%%%%%%%%%
\begin{prop}\label{main(1)'}
With the notation as above, the following three conditions are equivalent: 
\begin{enumerate}
\item $S$ is minimal. 
\item $\rho (S)=2$ and $S$ is quasi-minimal. 
\item  $\rho (S)=2$ and the type of $S$ is one of those in the list of Theorem \ref{main(1)}. 
\end{enumerate}
\end{prop}
%%%%%%%%%%%%%%%%%%%
\begin{remark}
Assume that Proposition \ref{main(1)'} is true and there exists a weak del Pezzo surface $S'$ with $\rho (S')=2$ such that $S$ and $S'$ have the same type. 
Then we see that $S$ is quasi-minimal if and only if the type of $S$ is one of those in the list of Theorem \ref{main(1)}. 
Indeed, by Proposition \ref{main(1)'}, $S'$ is quasi-minimal if and only if $S'$ is one of those in the list of Theorem \ref{main(1)}, moreover, since quasi-minimality depends on the type, $S'$ is quasi-minimal if and only if $S$ is quasi-minimal. 
\end{remark}
%%%%%%%%%%%%%%%%%%%
Let us prove Proposition \ref{main(1)'}. 
It is clear that {\rm (1)} implies {\rm (2)} in Proposition \ref{main(1)'}. 
Let us show that {\rm (2)} implies {\rm (3)} and {\rm (3)} implies {\rm (1)} in Proposition \ref{main(1)'}. 
In the case of $d=8$, it can be easily seen that these two implications hold, indeed, $S$ is always minimal since $S$ is a $k$-form of the Hirzebruch surface $\bF _2$ of degree two, i.e., $S_{\kc} \simeq \bF _2$. 
However, in the case of $d<8$, the proofs of these two implications are a bit long. 
Thus, we will give the proof for the case of $d<8$ in \S \S \ref{3-2}. 
%%%%%%%%%%%%%%%%%%%%%%%%%%%%%%%%%%%%%%%%%%%%%%%%%%%%%%%%%%%%%%%%%%%
\subsection{Proof of Proposition \ref{main(1)'}}\label{3-2}
With the notation as above, assume further $d \le 7$. 
%%%%%%%%%%%%%%%%%%%%%%%%%%%%%%%%%%%%%%%%%
\subsubsection{}\label{3-2-1}
In order to prove Proposition \ref{main(1)'}, we shall prepare some notation. 

We shall construct a birational morphism defined over $\kc$ from a weak del Pezzo surface $S_d$ of degree $d$ to the projective plane $\bP ^2_{\kc}$ by explicitly constructing the following composition: 
%%%%%%%%%%%%%%%%%%%
\begin{align}\label{sigma}
\sigma :S_d \overset{\sigma _{9-d}}{\to} S_{d+1} \overset{\sigma _{8-d}}{\to} \dots \overset{\sigma _2}{\to} S_8 \overset{\sigma _1}{\to} S_9=\bP ^2_{\kc}
\end{align}
%%%%%%%%%%%%%%%%%%%
such that $S_d$ and $S_{\kc}$ have the same type of $S_{\kc}$ (see Definition \ref{same type}), where $\sigma _i$ is a blow-up at a closed point for $i = 1,\dots ,9-d$. 
Notice that there exists such a birational morphism $\sigma :S_d\to \bP ^2_{\kc}$ by Lemma \ref{bir} and by the assumption $d \le 7$. 
In what follows, we shall take a composite of blowing-ups (\ref{sigma}). 
%%%%%%%%%%%%%%%%%%%

Let $e_0$ be the strict transform on $S_d$ of a general line on $\bP ^2_{\kc}$ and let $e_i$ be the total transform on $S_d$ of the exceptional divisor of $\sigma _i$ for $i=1,\dots ,9-d$. 
Then $\Pic (S_d)$ can be expressed as the free $\bZ$-module $I_d := \bigoplus _{i=0}^{9-d} \bZ e_i$ with a bilinear form generated by $(e_0^2)=1$, $(e_i^2)=-1$ for $i >0$ and $(e_i \cdot e_j)=0$ for $i,j \ge 0$ with $i \not= j$. 
%%%%%%%%%%%%%%%%%%%

Letting $M_1,\dots ,M_r$ be all $(-2)$-curves on $S_d$, we note that each $(-2)$-curve corresponds to one of the following element in $I_d$ (see {\cite[Proposition 8.2.7]{D12}}): 
%%%%%%%%%%%%%%%%%%%
\begin{align}\label{(-2)}
\begin{aligned}
&m^0_{i,j} := e_i - e_j& &(0<i<j \le 9-d,\ d \le 7); \\
&m^1_{i_1,i_2,i_3} := e_0 - (e_{i_1}+e_{i_2}+e_{i_3})& &(0< i_1 < i_2 < i_3 \le 9-d,\ d \le 6); \\
&m^2 := 2e_0 - (e_1+\dots +e_6)& &(d=3); \\
&m^2_{i_1,\dots, i_{3-d}} := 2e_0  - (e_{i_1}+\dots +e_{i_{3-d}})& &(0<i_1<\dots <i_{3-d} \le 9-d,\ d \le 2); \\
&m^3_i := 3e_0  -(e_1+\dots +e_8)- e_i& &(0 < i \le 9-d,\ d=1). 
\end{aligned}
\end{align}
%%%%%%%%%%%%%%%%%%%

Letting $k_d := -3e_0+e_1+\dots +e_{9-d} \in I_d$, which corresponds to the canonical divisor on $S_d$, we also note that any $e \in I_d$ satisfying $(e^2)=(e \cdot k_d)=-1$ is expressed as one of the following (see {\cite[Proposition 8.2.19]{D12}}): 
%%%%%%%%%%%%%%%%%%%
\begin{align}\label{(-1)}
\begin{aligned}
&e_i& &(0 <i \le 9-d,\ d \le 7); \\
&\ell _{i,j} := e_0-(e_i+e_j)& &(0 <i<j \le 9-d,\ d \le 7); \\
&2e_0-(e_{i_1}+\dots +e_{i_5})& &(0 <i_1 < \dots <i_5 \le 9-d,\ d \le 5); \\
&-k_2-e_i& &(0 <i \le 7,\ d=2); \\
&-k_1-e_i+e_j& &(0 <i,j \le 8,\ i \not= j,\ d=1); \\
&-k_1+e_0-(e_{i_1}+e_{i_2}+e_{i_3})& &(0 <i_1 < i_2 <i_3 \le 8,\ d=1); \\
&-k_1+2e_0-(e_{i_1}+\dots +e_{i_6})& &(0 <i_1 < \dots <i_6 \le 8,\ d=1); \\
&-2k_1-e_i& &(0 <i \le 8,\ d=1). 
\end{aligned}
\end{align}
%%%%%%%%%%%%%%%%%%%

By Lemma \ref{(-1)-curve}, the set of all $(-1)$-curves on $S_d$ has one-to-one correspondence to the set of all elements in $({\ref{(-1)}})$ which have non-negative intersection number with elements in $I_d$ corresponding to all $(-2)$-curves on $S_d$. 
Thus, we are able to see the intersection form of all $(-1)$-curves and $(-2)$-curves on $S_{\kc}$ as the surfaces $S_{\kc}$ and $S_d$ have the same type. 
In what follows, we will determine the quasi-minimality of $S$ by studying elements as in ({\ref{(-1)}}) and ({\ref{(-2)}}) according to the type of $S$. 
%%%%%%%%%%%%%%%%%%%
\begin{remark}
A $(-2)$-curve $M$, which corresponds to an element $m_{i_1,i_2,i_3}^1$ (resp. $m^2$ or $m^2_{i_1,\dots, i_{3-d}}$, $m^3_{i_1}$) in $I_d$, is a strict transform of a line (resp. an irreducible conic, an irreducible cubic with a singular point) by a blow-up at some points on $\bP ^2_{\kc}$, which may include infinitely near points. 
For instance, assuming that $M$ corresponds to $m_{i_1,i_2,i_3}^1$, this blow-up includes infinitely near points if and only if there exists a $(-2)$-curve on $S_d$ corresponding to $m_{i_1,i_2}^0$, $m_{i_1,i_3}^0$ or $m_{i_2,i_3}^0$ in $I_d$. 
\end{remark}
%%%%%%%%%%%%%%%%%%%%%%%%%%%%%%%%%%%%%%%%%
\subsubsection{}
Let us prove that (3) implies (1) in Proposition \ref{main(1)'}. 
Assume that $\rho (S)=2$ and the type of $S$ is one of those in the list of Theorem \ref{main(1)} other than the type of $d=8$. 

We shall take a composite of blowing-ups (\ref{sigma}) in such a way that elements $m_1,\dots ,m_r \in I_d$ corresponding to all $(-2)$-curves on $S_d$ are as in Table \ref{table1} according to the type of $S_d$ (for the notation of their elements, see (\ref{(-2)})), where ``all $(-2)$-curves'' in Table \ref{table1} mean all elements in $I_d$ corresponding to all $(-2)$-curves on $S_d$, respectively. 
%%%%%%%%%%%%%%%%%%%
\begin{table}
\caption{Configuration of all $(-2)$-curves. }\label{table1}
\begin{center}
 \begin{tabular}{|c|c|llll|} \hline
$d$ & Type & \multicolumn{4}{|c|}{all $(-2)$-curves} \\ \hline \hline

$4$ & $(2A_1)_<$ & $m^0_{1,2}$, & \hspace{-2.5mm}$m^1_{3,4,5}$ & & \\ \hline \hline

$2$ & $A_1$ & $m^2_1$ & & & \\ \hline
$2$ & $A_2$ & $m^1_{2,3,4}$, & \hspace{-2.5mm}$m^1_{5,6,7}$ & & \\ \hline
$2$ & $(4A_1)_>$ & $m^1_{2,3,4}$, & \hspace{-2.5mm}$m^1_{2,5,6}$, & \hspace{-2.5mm}$m^1_{3,5,7}$, & \hspace{-2.5mm}$m^1_{4,6,7}$ \\ \hline \hline

$1$ & $2A_1$ & $m^2_{1,2}$, & \hspace{-2.5mm}$m^3_2$ & & \\ \hline
$1$ & $2A_2$ & $m^1_{2,3,7}$, & \hspace{-2.5mm}$m^1_{4,5,8}$, & \hspace{-2.5mm}$m^1_{6,7,8}$, & \hspace{-2.5mm}$m^2_{7,8}$ \\ \hline
\end{tabular}
\end{center}
\end{table}
%%%%%%%%%%%%%%%%%%%
By construction, we see $\Pic (S_{\kc}) \simeq \Pic (S_d) \simeq I_d$ preserving the intersection form. 
Then we obtain the following claim: 
%%%%%%%%%%%%%%%%%%%
\begin{claim}\label{claim}
The following three assertions hold: 
\begin{enumerate}
\item For an arbitrary integer $i$ with $2 \le i \le 9-d$, there exist two $(-1)$-curves $E_{i,+}$ and $E_{i,-}$ on $S_d$ corresponding to elements $e_i$ and $\ell _{1,i}$ in $I_d$, respectively. 
\item If $d \ge 2$, then all $(-1)$-curves meeting at least one $(-2)$-curve on $S_d$ are only $E_{i,+}$ and $E_{i,-}$ for $2 \le i \le 9-d$. 
\item If $d=1$, then all $(-1)$-curves meeting at least two $(-2)$-curves on $S_d$ are only $E_{i,+}$ and $E_{i,-}$ for $2 \le i \le 9-d$. 
\end{enumerate}
\end{claim}
%%%%%%%%%%%%%%%%%%%
\begin{proof}
In (1), we shall check that intersection numbers $(e_i \cdot m_j)$ and $(\ell _{1,i} \cdot m_j)$ are non-negative for $2 \le i \le 9-d$ and $1 \le j \le r$, however, it is left to the reader since it can be easily shown by explicit computing. 

In (2) and (3), let $E$ be a $(-1)$-curve on $S_d$, let $e$ be an element in $I_d$ corresponding to $E$ and set $m := m_1+ \dots +m_r \in I_d$. 
Noting that $e$ is one of those in the list of (\ref{(-1)}), we shall calculate the intersection number $(e \cdot m)$ according to degree $d$: 

If $d=4$, then $m=e_0+e_1-(e_2+\dots +e_5)$, so that we have: 
\begin{align}\label{d=4}
(e \cdot m) = 
\left\{ \begin{array}{cl}
1 & \text{if}\ e= e_i\ \text{or}\ \ell _{1,i} \\
-1 & \text{otherwise} \\
\end{array} \right. \ (2 \le i \le 5).
\end{align}

If $d=2$ and $S$ is of $A_1$-type or $A_2$-type (resp. $(4A_1)_>$-type), then $m=2e_0-(e_2+\dots +e_7)$ (resp. $m=2\{ 2e_0-(e_2+\dots +e_7)\}$), so that we have: 
\begin{align}\label{d=2}
(e \cdot m) = 
\left\{ \begin{array}{cl}
1\ (\text{resp.}\ 2) & \text{if}\ e= e_i\ \text{or}\ \ell _{1,i} \\
-1\ (\text{resp.}\ -2) & \text{if}\ e=-k_2-e_i\ \text{or}\ -k_2-\ell _{1,i}\\
0 & \text{otherwise} \\
\end{array} \right. \ (2 \le i \le 7), 
\end{align}
where we note $-k_2-\ell _{1,i} = 2e_0 -(e_2+\dots +e_7)+e_i$ for $2 \le i \le 7$. 

If $d=1$, then $m=5e_0-e_1-2(e_2+\dots +e_8) = -2k_1-(e_0-e_1)$, so that we have: 
\begin{align}\label{d=1}
\begin{aligned}
(e_i \cdot m)&=\left\{ \begin{array}{cl}
2 & \text{if}\ i>1 \\
1 & \text{if}\ i=1 \\
\end{array} \right. 
\quad (1 \le i \le 8); \\ 
(\ell _{i,j} \cdot m)&=\left\{ \begin{array}{cl}
2 & \text{if}\ i=1 \\
1 & \text{if}\ i>1 \\
\end{array} \right. 
\quad (1 \le i<j \le 8); \\
(e \cdot m) &< 2 \quad \text{if}\ (e \cdot e_0) \ge 2, 
\end{aligned}
\end{align}
where we note $(e \cdot m) = (e \cdot -2k_1) -(e \cdot e_0-e_1)$ and $(e \cdot e_0-e_1) >0$ by ({\ref{(-1)}}) if $(e \cdot e_0) \ge 2$. 

Therefore, we obtain the assertions (2) and (3) of the claim. 
Indeed, if $d \ge 2$ (resp. $d=1$) and $E$ meets at least one $(-2)$-curve (resp. at least two $(-2)$-curves) on $S_d$, then
$E$ is $E_{i,+}$ or $E_{i,-}$ for some $2 \le i \le 9-d$ by virtue of (\ref{d=4}) and (\ref{d=2}) (resp. (\ref{d=1})). 
\end{proof}
%%%%%%%%%%%%%%%%%%%
Now, we shall prove that (3) implies (1) in Proposition \ref{main(1)'}. 
Let $S_d$ be the same as above. 
Let $D$ be the reduced union of $(-1)$-curves on $S_{\kc}$ corresponding to elements $e_i$ and $\ell _{1,i}$ in $I_d$ for $2 \le i \le 9-d$ in $I_d$. 
By Claim \ref{claim}, we see that $D$ is defined over $k$. 
Moreover, we have $(D^2)=\left( \sum _{i=2}^{9-d}(e_i+\ell _{1,i}) \right)^2 = (8-d)^2(e_0^2-e_1^2) = 0$. 

Suppose on the contrary that there exists a birational morphism $\tau : S \to V$ to smooth projective surface $V$ with $\rho (V)<\rho (S)$ defined over $k$. 
Then $V$ is a smooth del Pezzo surface of $\rho (V)=1$ (cf. {\cite[Theorem 9.3.20]{P17}}) by virtue of $\rho (S)=2$. 
Hence, there exists a $(-1)$-curve $E$ meeting at least one $(-2)$-curve on $S_{\kc}$ such that $\tau _{\kc}$ is a contraction of the $\Gal$-orbit of $E$. 
Notice that $E$ is not any irreducible component of $D$. 
Otherwise, we have $\tau _{\ast}(D) \not= 0$ and $(\tau _{\ast}(D)^2)=0$ by Claim \ref{claim}(1). This is a contradiction to the fact $\rho (V)=1$. 
Hence, we see that $d=1$ and $E$ meets only one $(-2)$-curve on $S_{\kc}$ by Claim \ref{claim}(2) and (3). 
Let $M_1,\dots ,M_r$ be all $(-2)$-curves on $S_{\kc}$, where $r=2$ (resp. $r=4$) if $S$ is of $2A_1$-type (resp. $2A_2$-type). 
Furthermore, let $s$ be the number of $(-1)$-curves on $S_{\kc}$, which meet a $(-2)$-curve $M_1$ and are contracted by $\tau _{\kc}$, where we note that $s$ is constant not depending on the way to take a $(-2)$-curve $M_1$ on $S_{\kc}$. 
Indeed, all $(-2)$-curves on $S_{\kc}$ lie on the same $\Gal$-orbit since $V_{\kc}$ does not contain any $(-2)$-curve. 
If $S$ is of $2A_1$-type (resp. $2A_2$-type), then the degree of $V_{\kc}$ is equal to $2s+1$ (resp. $4s+1$), which is not equal to $7$ and is at most $9$, and we obtain $0<(\tau _{\ast}(M_1+\dots +M_r)^2) = -4+2s$ (resp. $-4+4s$) by virtue of $\rho (V)=1$. 
Thus, $V_{\kc}$ is of degree $9$, namely, $V_{\kc} \simeq \bP ^2_{\kc}$. 
In particular, the self-intersection number of any irreducible curve on $V_{\kc}$ is a positive square number, however, $(\tau _{\kc ,\ast}(M_1)^2)$ is equal to $2$ (resp. $0$) if $S$ is of $2A_1$-type (resp. $2A_2$-type). It is a contradiction. 
Therefore, we see that $S$ must be $k$-minimal. 
%%%%%%%%%%%%%%%%%%%
\begin{remark}
Assuming $\rho (S)=2$, we obtain that $S$ is minimal by the above argument. 
Letting $S_d$ be the same as above, for any $2 \le i \le 9-d$, two $(-1)$-curves on $S_{\kc}$, which correspond to $e_i$ and $\ell _{1,i}$ respectively, lie the same $\Gal$-orbit. 
\end{remark}
%%%%%%%%%%%%%%%%%%%%%%%%%%%%%%%%%%%%%%%%%
\subsubsection{}
In order to prove that (2) implies (3) in Proposition \ref{main(1)'}, assume that the type of $S$ is one of those in the list of Lemma \ref{min(1)} such that it does not appear in the list of Theorem \ref{main(1)}. 
Then we shall show that $S$ is not quasi-minimal. 
%%%%%%%%%%%%%%%%%%%

At first, we deal with the case in which $S$ is of degree $d=1$ and of $A_1$-type. 
We can choose a composite of blowing-ups $(\ref{sigma})$ in such a way that $S_d = S_1$ contains only one $(-2)$-curve corresponding to $m_1^3 \in I_1$ (for the notation of $m_1^3$, see (\ref{(-2)})). 
Letting $e$ be one of those in the list of ({\ref{(-1)}}), we obtain that $(e \cdot m_1^3)=2$ if and only if $e=e_1$, indeed, 
$(e \cdot m_1^3)=(e \cdot -k_1)-(e \cdot e_1)=1-(e \cdot e_1)$. 
Since $S_1$ and $S_{\kc}$ have the same type, there exists a unique $(-1)$-curve $E$ satisfying $(E \cdot M)=2$ on $S_{\kc}$, where $M$ is the unique $(-2)$-curve on $S_{\kc}$. 
This means that $|\sM_E (2,2)|=1$ and there is no $(-1)$-curve $E'$ meeting $E$ on $S_{\kc}$ such that $|\sM_{E'}(2,2)|=1$. Hence $S$ is not quasi-minimal. 
%%%%%%%%%%%%%%%%%%%

In what follows, we deal with the remaining cases. 
As an example, we shall explain the case in which $S$ is of degree $d=2$ and of $(3A_1)_>$-type. 
Then $S_{\kc}$ contains exactly three $(-2)$-curves. 
Let us put $\alpha := 2$, where we notice that $\alpha$ is smaller than or equal to the number of $(-2)$-curves on $S_{\kc}$. 
Let $\beta$ be the number of $(-1)$-curves on $S_{\kc}$ meeting exactly $\alpha$-times of $(-2)$-curves on $S_{\kc}$. 
In order to determine the value of $\beta$, we shall take a composite of blowing-ups $(\ref{sigma})$ in such a way that $S_d=S_2$ contains exactly three $(-2)$-curves corresponding to $m_{1,2,3}^1,\, m_{1,3,4}^1,\, m_1^2 \in I_2$ (see (\ref{(-2)})). 
Then we see that elements in $I_2$ corresponding to all $(-1)$-curves meeting exactly $\alpha$-times of $(-2)$-curves on $S_2$ are only $e_1,\dots ,e_5$ and $\ell _{6,7}$ (see Example \ref{in323}). 
Hence, we obtain $\beta = 6$. 
Moreover, the union of $\beta$-times of $(-1)$-curves on $S_2$, which correspond to $e_1,\dots ,e_5$ and $\ell _{6,7}$ in $I_2$, is disjoint. 
Since $S_2$ and $S_{\kc}$ have the same type, letting $E$ be a $(-1)$-curve on $S_{\kc}$ corresponding to one of $e_1,\dots ,e_5$ or $\ell _{6,7}$ in $I_2$, we see that $|\sM _E(2,1)|=\alpha$ and there is no $(-1)$-curve $E'$ meeting $E$ on $S_{\kc}$ such that $|\sM_{E'}(2,1)| = \alpha$. 
Thus, $S$ is not quasi-minimal. 
%%%%%%%%%%%%%%%%%%%

The other cases can be shown by a similar argument, by changing the value of $\alpha$ and elements in $I_d$, which correspond to all $(-2)$-curves on $S_d$, according to the type of $S$. 
We will now explain how to do this. 
Let $\alpha$ be this as in Table \ref{table2} according to the type of $S$, and let us take a composite of blowing-ups $(\ref{sigma})$ in such a way that all $(-2)$-curves on $S_d$ corresponding to elements in $I_d$, which are these as in ``all $(-2)$-curves" in Table \ref{table2} according to the type of $S$. 
Then we see that elements in $I_2$, which correspond to all $(-1)$-curves meeting exactly $\alpha$-times of $(-2)$-curves on $S_d$, are only these as in ``$\beta$-times of $(-1)$-curves'' in Table \ref{table2} according to the type of $S$ (see Examples \ref{in323} and \ref{in323(2)} for how to find all elements in $I_d$). 
For instance, if $d=2$ and the type of $S$ is $(3A_1)_>$-type, then these elements yield $e_1,\dots ,e_5,\ell _{6,7} \in I_2$ as demonstrated above. 
Hence, $\beta$ is this as in Table \ref{table2} according to the type of $S$. 
Moreover, we see that the union of $\beta$-times of $(-1)$-curves, which meet exactly $\alpha$-times of $(-2)$-curves on $S_d$, is disjoint. 
Since $S_d$ and $S_{\kc}$ have the same type, letting $E$ be a $(-1)$-curve on $S_{\kc}$ corresponding to one of $\beta$-times of $(-1)$-curves meeting exactly $\alpha$-times of $(-2)$-curves on $S_d$, we see that $|\sM _E(2,1)|=\alpha$ and there is no $(-1)$-curve $E'$ meeting $E$ on $S_{\kc}$ such that $|\sM _{E'}(2,1)| = \alpha$. 
Thus, $S$ is not quasi-minimal. 
%%%%%%%%%%%%%%%%%%%

In summary, we show that $S$ is quasi-minimal if the type of $S$ is one of those in the list of Theorem \ref{main(1)}. 
%%%%%%%%%%%%%%%%%%%
\begin{remark}
In the above argument, we do not actually use the assumption $\rho (S)=2$. 
\end{remark}
%%%%%%%%%%%%%%%%%%%
\begin{table}[p]
\caption{The value of $\beta$ and configuration of $\beta$-times of $(-1)$-curves. }\label{table2} 
\begin{center}
 \begin{tabular}{|c|c|c|llll||c|l|}
 \hline
$d$ & Type & $\alpha$ & \multicolumn{4}{|c||}{all $(-2)$-curves} & $\beta$ & \multicolumn{1}{|c|}{$\beta$-times of $(-1)$-curves} \\ \hline \hline

$7$ & $A_1$ & $1$ & $m^0_{1,2}$ & & & & $1$ & $e_2$ \\ \hline \hline

$6$ & $A_2$ & $1$ & $m^0_{1,2}$, & \hspace{-2.5mm}$m^0_{2,3}$ & & & $2$ & $e_3$, $\ell _{1,2}$  \\ \hline
$6$ & $2A_1$ & $2$ & $m^0_{1,2}$, & \hspace{-2.5mm}$m^1_{1,2,3}$ & & & $1$ & $e_2$ \\ \hline
$6$ & $(A_1)_<$ & $1$ & $m^1_{1,2,3}$ & & & & $3$ & $e_1$, $e_2$, $e_3$ \\ \hline
$6$ & $(A_1)_>$ & $1$ & $m^0_{1,2}$ & & & & $2$ & $e_2$, $\ell _{1,3}$ \\ \hline \hline

$5$ & $A_2$ & $1$ & $m^0_{1,2}$, & \hspace{-2.5mm}$m^1_{1,3,4}$ & & & $3$ & $e_2$, $e_3$, $e_4$ \\ \hline
$5$ & $2A_1$ & $2$ & $m^0_{1,2}$, & \hspace{-2.5mm}$m^1_{1,2,3}$ & & & $1$ & $e_2$ \\ \hline
$5$ & $A_1$ & $1$ & $m^1_{1,2,3}$ & & & & $3$ & $e_1$, $e_2$, $e_3$ \\ \hline \hline

$4$ & $4A_1$ & $2$ & $m^0_{1,2}$, & \hspace{-2.5mm}$m^0_{3,4}$, & \hspace{-2.5mm}$m^1_{1,2,5}$, & \hspace{-2.5mm}$m^1_{3,4,5}$ & $4$ & $e_2$, $e_4$, $e_5$, $\ell _{1,3}$ \\ \hline
$4$ & $3A_1$ & $2$ & $m^0_{1,2}$, & \hspace{-2.5mm}$m^0_{3,4}$, & \hspace{-2.5mm}$m^1_{1,2,5}$ & & $2$ & $e_2$, $\ell _{1,3}$ \\ \hline
$4$ & $A_2$ & $1$ & $m^1_{1,2,3}$, & \hspace{-2.5mm}$m^0_{4,5}$ & & & $4$ & $e_1$, $e_2$, $e_3$, $\ell _{4,5}$ \\ \hline
$4$ & $(2A_1)_>$ & $2$ & $m^1_{1,2,3}$, & \hspace{-2.5mm}$m^1_{1,4,5}$ & & & $1$ & $e_1$ \\ \hline
$4$ & $A_1$ & $1$ & $m^1_{1,2,3}$ & & & & $4$ & $e_1$, $e_2$, $e_3$, $\ell _{4,5}$ \\ \hline \hline

\multirow{2}{*}{$3$} & \multirow{2}{*}{$3A_2$} & \multirow{2}{*}{$2$} & $m^0_{1,2}$, & \hspace{-2.5mm}$m^1_{1,3,4}$, & \hspace{-2.5mm}$m^0_{3,4}$, & \hspace{-2.5mm}$m^1_{3,5,6}$, & \multirow{2}{*}{$3$} & \multirow{2}{*}{$e_2$, $e_4$, $e_6$} \\
 & & & $m^0_{5,6}$, & \hspace{-2.5mm}$m^1_{1,2,5}$ & & & & \\ \hline
$3$ & $2A_2$ & $2$ & $m^0_{1,2}$, & \hspace{-2.5mm}$m^1_{1,5,6}$, & \hspace{-2.5mm}$m^0_{3,4}$, & \hspace{-2.5mm}$m^1_{1,2,3}$ & $1$ & $e_2$  \\ \hline
$3$ & $4A_1$ & $2$ & $m^1_{1,2,3}$, & \hspace{-2.5mm}$m^1_{1,4,5}$, & \hspace{-2.5mm}$m^1_{2,4,6}$, & \hspace{-2.5mm}$m^1_{3,5,6}$ & $6$ & $e_1$, $e_2$, $e_3$, $e_4$, $e_5$, $e_6$ \\ \hline
$3$ & $3A_1$ & $2$ & $m^1_{1,2,4}$, & \hspace{-2.5mm}$m^1_{1,3,5}$, & \hspace{-2.5mm}$m^1_{2,3,6}$ & & $3$ & $e_1$, $e_2$, $e_3$ \\ \hline
$3$ & $A_2$ & $1$ & $m^1_{1,2,3}$, & \hspace{-2.5mm}$m^1_{4,5,6}$ & & & $6$ & $e_1$, $e_2$, $e_3$, $e_4$, $e_5$, $e_6$ \\ \hline
$3$ & $2A_1$ & $2$ & $m^1_{1,2,3}$, & \hspace{-2.5mm}$m^1_{1,4,5}$ & & & $1$ & $e_1$ \\ \hline
$3$ & $A_1$ & $1$ & $m^2$ & & & & $6$ & $e_1$, $e_2$, $e_3$, $e_4$, $e_5$, $e_6$ \\ \hline \hline

\multirow{2}{*}{$2$} & \multirow{2}{*}{$3A_2$} & \multirow{2}{*}{$2$} & $m^0_{1,2}$, & \hspace{-2.5mm}$m^1_{1,3,4}$, & \hspace{-2.5mm}$m^0_{3,4}$, & \hspace{-2.5mm}$m^1_{3,5,6}$, & \multirow{2}{*}{$6$} & \multirow{2}{*}{$e_2$, $e_4$, $e_6$, $\ell _{1,7}$, $\ell _{3,7}$, $\ell _{5,7}$} \\
 & & & $m^0_{5,6}$, & \hspace{-2.5mm}$m^1_{1,2,5}$ & & & & \\ \hline
\multirow{2}{*}{$2$} & \multirow{2}{*}{$6A_1$} & \multirow{2}{*}{$3$} & $m^1_{1,2,5}$, & \hspace{-2.5mm}$m^1_{1,3,6}$, & \hspace{-2.5mm}$m^1_{1,4,7}$, & \hspace{-2.5mm}$m^1_{2,3,7}$, & \multirow{2}{*}{$4$} & \multirow{2}{*}{$e_1$, $e_2$, $e_3$, $e_4$} \\
 & & & $m^1_{2,4,6}$, & \hspace{-2.5mm}$m^1_{3,4,5}$ & & & & \\ \hline
$2$ & $5A_1$ & $3$ & \multicolumn{4}{|l||}{$m^1_{1,2,3}$, $m^1_{1,4,5}$, $m^1_{1,6,7}$, $m^1_{2,4,6}$, $m^1_{2,5,7}$} & $2$ & $e_1$, $e_2$ \\ \hline
$2$ & $2A_2$ & $2$ & $m^0_{1,2}$, & \hspace{-2.5mm}$m^1_{1,3,7}$, & \hspace{-2.5mm}$m^1_{1,2,6}$, & \hspace{-2.5mm}$m^1_{3,4,5}$ & $2$ & $e_2$, $e_3$  \\ \hline
$2$ & $(4A_1)_<$ & $3$ & $m^1_{1,2,3}$, & \hspace{-2.5mm}$m^1_{1,4,5}$, & \hspace{-2.5mm}$m^1_{1,6,7}$, & \hspace{-2.5mm}$m^1_{2,4,6}$ & $1$ & $e_1$ \\ \hline
$2$ & $(3A_1)_<$ & $3$ & $m^1_{1,2,3}$, & \hspace{-2.5mm}$m^1_{1,4,5}$, & \hspace{-2.5mm}$m^1_{1,6,7}$ & & $1$ & $e_1$ \\ \hline
$2$ & $(3A_1)_>$ & $2$ & $m^1_{1,2,3}$, & \hspace{-2.5mm}$m^1_{1,4,5}$, & \hspace{-2.5mm}$m^2_1$ & & $6$ & $e_1$, $e_2$, $e_3$, $e_4$, $e_5$, $\ell _{6,7}$ \\ \hline
$2$ & $2A_1$ & $2$ & $m^1_{1,2,3}$, & \hspace{-2.5mm}$m^2_3$ & & & $2$ & $e_1$, $e_2$ \\ \hline \hline

\multirow{2}{*}{$1$} & \multirow{2}{*}{$4A_2$} & \multirow{2}{*}{$3$} & 
 $m^1_{1,3,4}$, & \hspace{-2.5mm}$m^1_{2,5,8}$, & \hspace{-2.5mm}$m^1_{1,5,6}$, & \hspace{-2.5mm}$m^1_{2,4,7}$, & \multirow{2}{*}{$8$} & \multirow{2}{*}{$e_1$, $e_2$, $e_3$, $e_4$, $e_5$, $e_6$, $e_7$, $e_8$} \\
 & & & $m^1_{1,7,8}$, & \hspace{-2.5mm}$m^1_{2,3,6}$ &  \hspace{-2.5mm}$m^1_{3,5,7}$, &  \hspace{-2.5mm}$m^1_{4,6,8}$ & & \\ \hline
\multirow{2}{*}{$1$} & \multirow{2}{*}{$3A_2$} & \multirow{2}{*}{$3$} & $m^1_{1,3,4}$, & \hspace{-2.5mm}$m^1_{2,5,8}$, & \hspace{-2.5mm}$m^1_{1,5,6}$, & \hspace{-2.5mm}$m^1_{2,4,7}$, & \multirow{2}{*}{$2$} & \multirow{2}{*}{$e_1$, $e_2$} \\
 & & & $m^1_{1,7,8}$, & \hspace{-2.5mm}$m^1_{2,3,6}$ & & & & \\ \hline
\multirow{2}{*}{$1$} & \multirow{2}{*}{$6A_1$} & \multirow{2}{*}{$4$} & $m^1_{1,2,4}$, & \hspace{-2.5mm}$m^1_{1,3,5}$, & \hspace{-2.5mm}$m^1_{2,3,6}$, & \hspace{-2.5mm}$m^2_{1,6}$, & \multirow{2}{*}{$3$} & \multirow{2}{*}{$e_1$, $e_2$, $e_3$} \\
 & & & $m^2_{2,5}$, & \hspace{-2.5mm}$m^2_{3,4}$ & & & & \\ \hline
$1$ & $5A_1$ & $4$ & \multicolumn{4}{|l||}{$m^1_{1,3,6}$, $m^1_{1,4,5}$, $m^1_{2,3,4}$, $m^2_{3,5}$, $m^2_{4,6}$} & $1$ & $e_1$ \\ \hline
$1$ & $(4A_1)_<$ & $3$ & $m^2_{1,2}$, & \hspace{-2.5mm}$m^2_{3,4}$, & \hspace{-2.5mm}$m^2_{5,6}$, & \hspace{-2.5mm}$m^2_{7,8}$ & $8$ & $e_1$, $e_2$, $e_3$, $e_4$, $e_5$, $e_6$, $e_7$, $e_8$ \\ \hline
$1$ & $(4A_1)_>$ & $4$ & $m^1_{1,2,3}$, & \hspace{-2.5mm}$m^1_{1,4,5}$, & \hspace{-2.5mm}$m^1_{1,6,7}$, & \hspace{-2.5mm}$m^3_8$ & $1$ & $e_1$ \\ \hline
$1$ & $3A_1$ & $3$ & $m^1_{1,2,3}$, & \hspace{-2.5mm}$m^1_{1,4,5}$, & \hspace{-2.5mm}$m^1_{1,6,7}$ & & $1$ & $e_1$ \\ \hline
$1$ & $A_2$ & $2$ & $m^1_{1,2,3}$, & \hspace{-2.5mm}$m^2_{2,3}$ & & & $1$ & $e_1$ \\ \hline
\end{tabular}
 \end{center}
\end{table}
%%%%%%%%%%%%%%%%%%%
%Here, in Table \ref{table2},  ``all $(-2)$-curves''  means all elements in $I_d$, which correspond all $(-2)$-curves on $S_d$. 
%Moreover, ``$\beta$-times of $(-1)$-curves'' means all elements in $I_d$, which correspond all $(-1)$-curves meeting exactly $\alpha$-times of $(-2)$-curves respectively on $S_d$, if we take the value of $\alpha$ and ``all $(-2)$-curves'' such as in Table \ref{table2}. 
%%%%%%%%%%%%%%%%%%%
\begin{eg}\label{in323}
Assume that $S$ is of degree $d=2$ and of $(3A_1)_>$-type. Then $S_{\kc}$ contains exactly three $(-2)$-curves. 
Let us put $\alpha := 2$ and let us choose a composite of blowing-ups $(\ref{sigma})$ in such a way that $S_d=S_2$ contains exactly three $(-2)$-curves corresponding to $m^1_{1,2,3},\, m^1_{1,4,5},\, m^2_1 \in I_2$ (see Table \ref{table2}). 
Then we shall determine all elements in $I_2$ corresponding to all $(-1)$-curves meeting exactly two $(-2)$-curves on $S_{\kc}$. 
At first, we can easily check that intersection numbers $(e \cdot m_{1,2,3}^1)$, $(e\cdot m_{1,4,5}^1)$ and $(e \cdot m^2_1)$ are equal to $0$ or $1$ for any $e=e_1,\dots ,e_5$, $\ell _{6,7}$. 
Next, we put $m := m^1_{1,2,3}+m^1_{1,4,5}+m^2_1$ and determine any element $e \in I_2$ as in (\ref{(-1)}) satisfying $(e \cdot m) = \alpha (=2)$. In consideration of $m=4e_0-2(e_1+\dots +e_5)-(e_6+e_7)$, we can calculate as follows: 
\begin{itemize}
\item If $e=e_i$ $(1 \le i \le 7)$, then $(e \cdot m)=2$ if and only if $1 \le i \le 5$. 
\item If $e=\ell _{i,j}$ $(1 \le i<j \le 7)$, then $(e \cdot m)=2$ if and only if $(i,j)=(6,7)$. 
\item If $e=2e_0-(e_{i_1}+\dots +e_{i_5})$ $(1 \le i_1<\dots <i_5 \le 7)$, then $(e \cdot m)\le 8-2 \cdot 1 - 3 \cdot 2=0 < 2$. 
\item If $e=-k_1-e_i$ $(1 \le i \le 7)$, then $(e \cdot m)\le -1<2$. 
\end{itemize}
Thus, we certainly see that all elements in $I_2$, which correspond to all $(-1)$-curves meeting exactly $\alpha (=2)$-times of $(-2)$-curves on $S_{\kc}$, are exhausted by $e_1,\dots ,e_5$ and $\ell _{6,7}$. 
For the other cases with $d \ge 2$, we can calculate in a similar way. 
\end{eg}
%%%%%%%%%%%%%%%%%%%
The following deals with all cases of $d=1$: 
%%%%%%%%%%%%%%%%%%%
\begin{eg}\label{in323(2)}
Assume that $S$ is of degree $d=1$. 
We shall take a composite of blowing-ups $(\ref{sigma})$ in such a way that elements $m_1,\dots ,m_r \in I_1$ corresponding to all $(-2)$-curves on $S_1$ are as in Table \ref{table2}, according to the type of $S$. 
Then the element $m := m_1+\dots +m_r$ is expressed as follows depending on the types of $S$: 
\begin{itemize}
\item $4A_2$ or $(4A_1)_<$: $(3\alpha -1)e_0-\alpha (e_1+\dots +e_8)$,\ \text{where}\ $\alpha = 3$, $\beta =8$; 
\item $5A_1$: $(3\alpha -5)e_0-\alpha e_1-(\alpha -1)(e_2+e_3+e_4)-(\alpha -2)(e_5+\dots +e_8)$,\ \text{where}\ $\alpha = 4$, $\beta =1$; 
\item Otherwise: $3\alpha 'e_0-\alpha (e_1+\dots +e_{\beta})-\alpha '(e_{\beta +1}+\dots +e_8)$,\ \text{where} $\alpha ' < \alpha $. 
\end{itemize}
Hence, letting $e$ be one of those in the list of ({\ref{(-1)}}), if $(e \cdot m)=\alpha$, then we see that $(e \cdot e_0)=0$, i.e., $e=e_i$ for some $1 \le i \le 8$. 
Indeed, assuming $(e \cdot e_0)>0$, we have $(e \cdot m) < (e \cdot -\alpha k_1) = \alpha$ by noting $(e \cdot e_i) \ge 0$ $(1 \le i \le 8)$. 
Moreover, we see that $(e_i \cdot m) =\alpha$ if and only if $1 \le i \le \beta$. Obviously we obtain $(e_i \cdot m_j) \ge 0$ for $1 \le i \le \beta$ and $1 \le j \le r$. 
\end{eg}
%%%%%%%%%%%%%%%%%%%%%%%%%%%%%%%%%%%%%%%%%%%%%%%%%%%%%%%%%%%%%%%%%%%%%%%%%%%%%%%%%%%%%%%%%%%%%%%%%%%%%%%%%%%
\section{Proof of Theorem \ref{main(2)}}\label{4}
%%%%%%%%%%%%%%%%%%%
In this section, we shall prove Theorem \ref{main(2)}. 
Let $S$ be a minimal weak del Pezzo surface of degree $d$ and $\rho (S)=2$ over a field $k$ of characteristic zero. 
By Theorem \ref{main(1)}, we see that Theorem \ref{main(2)} is a consequence of the following proposition: 
%%%%%%%%%%%%%%%%%%%
\begin{prop}\label{prop}
With the notation as above, the following three assertions hold true: 
\begin{enumerate}
\item If $d=8$, then $S$ contains an $\bA ^1_k$-cylinder if and only if there exists a conic bundle $\pi :S \to B$, which  admits a section defined over $k$. 
\item If $d=8$, then $S$ contains the affine plane $\bA ^2_k$ if and only if $S(k) \not= \emptyset$. 
\item If $d<8$, then $S$ does not contain any $\bA ^1_k$-cylinder. 
\end{enumerate}
\end{prop}
%%%%%%%%%%%%%%%%%%%
We will prove Proposition \ref{prop} according to the degree $d$ of $S$. 
More precisely, Proposition \ref{prop}(1) and (2) will be shown in \S \S \ref{4-1} and Proposition \ref{prop}(3) will be shown in \S \S \ref{4-2}. 
%%%%%%%%%%%%%%%%%%%%%%%%%%%%%%%%%%%%%%%%%%%%%%%%%%%%%%%%%%%%%%%%%%%
\subsection{Case of degree $8$}\label{4-1}
In this subsection, we shall show Proposition \ref{prop}(1) and (2). 
Let us assume $d=8$. 
Then $S$ is a $k$-form of $\bP ^1_{\kc} \times \bP ^1_{\kc}$ or the Hirzebruch surface $\bF _2$ of degree two, i.e., $S_{\kc} \simeq \bP ^1_{\kc} \times \bP ^1_{\kc}$ or $S_{\kc} \simeq \bF _2$. 
Moreover, $S$ is endowed with a structure of Mori conic bundle $\pi :S \to B$ such that the base extension of $\pi$ to the algebraic closure $\pi_{\kc} : S_{\kc} \to B_{\kc}$ is a $\bP^1$-bundle over $B_{\kc} \simeq \bP_{\kc}^1$ by Lemma \ref{bdl}. 

We shall consider the following three conditions: 
%%%%%%%%%%%%%%%%%%%
\begin{enumerate}
\item[{\rm (a)}] $S$ contains an $\bA ^1_k$-cylinder. 
\item[{\rm (b)}] There exists a Mori conic bundle $\pi : S \to B$, which admits a section defined over $k$. 
\item[{\rm (c)}] $S(k) \not= \emptyset$. 
\end{enumerate}
%%%%%%%%%%%%%%%%%%%
Then the following three lemmas hold: 
%%%%%%%%%%%%%%%%%
\begin{lemma}\label{lem4-1(1)}
{\rm (c)} implies {\rm (b)}. 
\end{lemma}
%%%%%%%%%%%%%%%%%
\begin{proof}
Noting $S(k) \not= \emptyset$ and $\rho _k(S)=2$, we see that $S \simeq \bP ^1_k \times \bP ^1_k$ or $S$ is the Hirzebruch surface of degree two defined over $k$ (i.e., $S \simeq \bP (\sO _{\bP ^1_k} \oplus \sO _{\bP ^1_k}(2))$) by using {\cite[Proposition 4.5.10]{P17}}. 
In particular, there exists a $\bP ^1$-bundle $S \to \bP ^1_k$ over $k$, which admits a section defined over $k$. 
\end{proof}
%%%%%%%%%%%%%%%%%
\begin{lemma}\label{lem4-1(2)}
{\rm (a)} implies {\rm (b)}. 
\end{lemma}
%%%%%%%%%%%%%%%%%
\begin{proof}
Suppose that $S$ contains an $\bA ^1_k$-cylinder, say $U \simeq \bA ^1_k \times Z$, and there is no Mori conic bundle, which admits a section defined over $k$. 
The closures in $S$ of fibers of the projection $pr_Z:U \simeq \bA ^1_k \times Z \to Z$ yields a linear system, say $\fl$, on $S$.  
Hence, we obtain the rational map $\Phi _{\fl}: S \dashrightarrow \overline{Z}$ associated with $\fl$, where $\overline{Z}$ is the smooth projective model of $Z$. 
If $\Phi _{\fl}$ is a morphism, then $\Phi _{\fl}$ is a Mori conic bundle, which admits a section defined over $k$ and is contained in $S \backslash U$, by Lemma \ref{MCB}. 
It is a contradiction to the assumption. 
Hence, $\fl$ is not base point-free.  
Then the base extension of $\fl$, say $\fl _{\kc}$, is not also base point-free. 
Since fibers of the base extension $pr_{Z_{\kc}}: U_{\kc} \simeq \bA ^1_{\kc} \times Z_{\kc} \to Z_{\kc}$ are isomorphic to the affine line, in particular, having only one-place at infinity, $\Bs (\fl _{\kc})$ is composed of one point. 
Furthermore, this point is defined over $k$. 
Thus, $\Bs (\fl )$ consists of only one $k$-rational point, which contradicts Lemma \ref{lem4-1(1)}. 
\end{proof}
%%%%%%%%%%%%%%%%%
\begin{lemma}\label{lem4-1(3)}
{\rm (b)} implies {\rm (a)}. 
\end{lemma}
%%%%%%%%%%%%%%%%%
\begin{proof}
By the assumption, we can take a Mori conic bundle $\pi :S \to B$, which admits a section defined over $k$, and let $\Gamma$ be a section of $\pi$ defined over $k$. 
As $\pi$ itself is defined over $k$, the base curve $B_{\kc}$ is also equipped with an action of $\Gal$ induced from that on $S_{\kc}$. 
The complement, say $U'$, of a divisor composed of $\Gamma$ and the pull-back by $\pi_{\kc}$ of a $\Gal$-orbit on $B_{\kc}$ is then a smooth affine surface defined over $k$. 
The restriction $\varphi := \pi|_{U'}$ of $\pi$ to $U'$ yields a morphism over an affine curve $Z' \subseteq B$. 
By construction, the base extension $\varphi _{\kc}$ is an $\bA ^1$-bundle to conclude that so is $\varphi$ by \cite[Theorem 1]{KM78}, which implies that there exists an open subset $Z \subseteq Z'$ such that $\varphi^{-1}(Z) \simeq \bA ^1_k \times Z$. 
This completes the proof. 
\end{proof}
%%%%%%%%%%%%%%%%%
Proposition \ref{prop}(1) follows from Lemmas \ref{lem4-1(2)} and \ref{lem4-1(3)}. 
%%%%%%%%%%%%%%%%%
\begin{cor}\label{cor4-1}
Let the notation be the same as above. 
If $-K_S$ is not ample, i.e., $S$ is a $k$-form of the Hirzebruch surface $\bF _2$ of degree two, then $S$ always contains an $\bA ^1_k$-cylinder. 
\end{cor}
%%%%%%%%%%%%%%%%%
\begin{proof}
By the assumption, $S_{\kc} \simeq \bF _2$ contains exactly one minimal section $M$, in particular, $M$ is defined over $k$. 
On the other hand, the base extension of the Mori conic bundle $\pi :S \to B$ to the algebraic closure $\pi _{\kc} : S_{\kc} \simeq \bF _2 \to B _{\kc}$ is a $\bP ^1$-bundle over $B_{\kc} \simeq \bP ^1_{\kc}$, whose $M$ is a section over $\kc$. 
Hence, $M$ is a section of $\pi$. 
This completes the proof by Proposition \ref{prop}(1). 
\end{proof}
%%%%%%%%%%%%%%%%%
Next, we will show Proposition \ref{prop}(2) as follows: 
%%%%%%%%%%%%%%%%%%%
\begin{proof}[Proof of Proposition \ref{prop}(2)]
Assume that $S$ admits a $k$-rational point. 
Let $\pi :S \to B$ be a Mori conic bundle. 
Then the base $B$ is a geometrically rational curve admitting a $k$-rational point to conclude that $B$ is isomorphic to $\bP_k^1$. Thus, $S$ contains the affine plane $\bA _k^2$. 
The converse direction is obvious. 
\end{proof}
%%%%%%%%%%%%%%%%%%%
\begin{eg}\label{not rational}
Take a smooth conic without ${\bR}$-rational points: 
\begin{align*}
C := \left( \, x^2 + y^2 + z^2 = 0 \, \right) \, \subseteq \, {\bP}_{\bR}^2 ={\rm Proj} (\bR [x,y,z]) 
\end{align*}
and let us put $S := C \times C$. 
Then $S$ is a $\bR$-form of $\bP ^1_{\bC} \times \bP ^1_{\bC}$ such that $S(\bR ) = \emptyset$. 
Hence, $S$ does not contain the affine plane $\bA ^2_{\bR}$ by Proposition \ref{prop}(2). 
But on the other hand, $S$ contains an $\bA ^1_{\bR}$-cylinder.
This can be shown as follows. 
Let $\varphi :S' \to S$ be the blow-up at a pair of conjugate points $x$ and $\bar{x}$. 
Then there exists a contraction $\psi :S' \to C \times \bP ^1_{\bR}$ of a disjoint union of two $(-1)$-curves defined over $\bR$ (see {\cite[Lemma 3.2]{K}} or {\cite{R02}}). 
Namely, we can take a fiber $F$ defined over $\bR$ of the second projection $pr _2: C_{\bC} \times \bP ^1_{\bC} \to \bP ^1_{\bC}$ such that $\Gamma := (\varphi \circ \psi^{-1})_{\ast}(F)$ is an irreducible curve on $S$ passing through $x$ and $\bar{x}$ defined over $\bR$. 
In fact, $\Gamma$ is a section defined over $\bR$ of the first and second projections $S \simeq C \times C \to C$, which are Mori conic bundles, respectively. 
Thus, the assertion follows from Proposition \ref{prop}(1). 

Incidentally, $C \times \bP ^1_{\bR}$ clearly contains an $\bA ^1_{\bR}$-cylinder. 
Hence, by Corollary \ref{cor4-1} combined with the classification of $\bR$-forms of $\bP ^1_{\bC} \times \bP ^1_{\bC}$ ({\cite[Lemma 1.16]{K}} or {\cite[Proposition 1.2]{R02}}), we know that  any minimal weak del Pezzo surface of degree $8$ defined over $\bR$ always contains an $\bA ^1_{\bR}$-cylinder. 
\end{eg}
%%%%%%%%%%%%%%%%%%%%%%%%%%%%%%%%%%%%%%%%%%%%%%%%%%%%%%%%%%%%%%%%%%%
\subsection{Case of degree less than $8$}\label{4-2}
In this subsection, let us assume $d<8$. 
By Lemma \ref{bdl}, $S$ is endowed with a structure of Mori conic bundle $\pi :S \to B$ such that $\pi _{\kc}$ admits a singular fiber. 
Notice that $B$ is isomorphic to $\bP_k^1$ provided that $S$ admits a $k$-rational point. 
The purpose of this subsection is to prove Proposition \ref{prop}(3). 
In other words, we shall show that $S$ does not contain any $\bA _k^1$-cylinder. 

The following lemma plays an important role in what follows and is the key lemma for the proof of Theorem \ref{main(2)}. 
Noting that we need to treat a minimal weak del Pezzo surface $S$ with $\rho (S)=2$, this lemma will be proved by the argument of {\cite[Proposition 9]{DK18}}, which deals with del Pezzo surfaces of Picard rank one, combined with {\it the variant of Corti's inequality} (see \cite[Theorem 3.1]{C00}): 
%%%%%%%%%%%%%%%%%%%
\begin{lemma}\label{Corti lem}
With the notation as above, let  $\fl$ be a linear pencil on $S$ such that $\Bs (\fl )$ consists of only one $k$-rational point, say $x$. 
Assume that a general member $L$ of $\fl$ satisfies $L \backslash \{ x\} \simeq \bA ^1_k$ and is $\bQ$-linearly equivalent to $a(-K_S) +  bF$ for some $a,b \in \bQ$, where $F$ is the fiber of the conic bundle $\pi :S \to \bP ^1_k$ passing through $x$. 
Then $b$ must be negative. 
%Let $\mathfrak{l}$ be a linear pencil on $S$ such that $\Bs (\mathfrak{l})$ consist of only one point $x$, which is $k$-rational. 
%That is, we obtain the rational map $\Phi _{\mathfrak{l}}:S \dashrightarrow \bP ^1_k$ associate to $\mathfrak{l}$. 
%Futhremore, assume that a general member $L$ on $\mathfrak{l}$ satisfies $L \backslash \{ x\} \simeq \bA ^1_k$, where $L$ is a member of $\mathfrak{l}_{\kc}$ and is deinfed over $k$, and is $\bQ$-linearly equivalent to $a(-K_S) +  bF$ for some $a,b \in \bQ$, where $F$ is a fiber of a conic bundle $\pi :S \to \bP ^1_k$, which passes through $x$. 
%Then $b$ must be negative. 
\end{lemma}
%%%%%%%%%%%%%%%%%%%
\begin{proof}
Suppose $b \ge 0$. 
Note that $a$ must be positive by $0 \le (\fl \cdot F) = 2a$ and $0 < (\fl ^2) = a(da+4b)$. 
Let $\Phi _{\fl}:S \dashrightarrow \bP ^1_k$ be the rational map associate to $\fl$, and let $\psi:\widetilde{S} \to S$ be the shortest succession of blow-ups $x \in \Bs (\fl )$ and its infinitely near points such that the strict transform $\widetilde{\fl} := \psi ^{-1}_{\ast}\fl$ of $\fl$ is free of base points to give rise to a morphism $\widetilde{\varphi} := \Phi _{\fl} \circ \psi$ (see the following diagram): 
%%%%%%%%%%%%%%%%%%%
\begin{align*}
\xymatrix{
S\ar@{.>}[r]^-{\Phi _{\fl}} & \bP ^1_k\\
\widetilde{S} \ar[u]^-{\psi} \ar[ru]_-{\widetilde{\varphi}}
}
\end{align*}
%%%%%%%%%%%%%%%%%%%
Notice that $\psi$ is defined over $k$ by construction. 
Letting ${\{ \widetilde{E}_i \}}_{1\le i \le n}$ be the exceptional divisors of $\psi$ with $\widetilde{E}_n$ the last exceptional one, which is a section of $\widetilde{\varphi}$, we have: 
%%%%%%%%%%%%%%%%%%%
\begin{align}\label{en}
(\widetilde{\fl} \cdot \widetilde{E}_i) = \left\{ \begin{array}{ll} 0 & (1 \le i \le n-1) \\ 1 & (i=n) \end{array} \right. 
\end{align}
%%%%%%%%%%%%%%%%%%%
and
%%%%%%%%%%%%%%%%%%%
\begin{align}\label{lc}
K_{\widetilde{S}} - \frac{b}{a} \psi ^{\ast}F + \frac{1}{a} \widetilde{\fl} = \psi ^{\ast} \left( K_S - \frac{b}{a}F + \frac{1}{a} \fl \right) + \sum _{i=1}^n\gamma _i \widetilde{E}_i
\end{align}
%%%%%%%%%%%%%%%%%%%
for some rational numbers $\gamma _1,\dots ,\gamma _n$. 
As $a>0$, $b \ge 0$ and $(\widetilde{\fl} ^2) = 0$, we have:
%%%%%%%%%%%%%%%%%%%
\begin{align*}
-2 &= (\widetilde{\fl} \cdot K_{\widetilde{S}}) \\
&= \left( \widetilde{\fl} \cdot K_{\widetilde{S}} + \frac{1}{a}\widetilde{\fl} \right) \\
&\ge \left( \widetilde{\fl} \cdot K_{\widetilde{S}} - \frac{b}{a} \psi ^{\ast}F   + \frac{1}{a}\widetilde{\fl} \right) \\
&\underset{(\ref{lc})}{=} \left( \widetilde{\fl} \cdot \psi ^{\ast} \left( K_S - \frac{b}{a}F + \frac{1}{a} \fl \right) \right) + \sum _{i=1}^n\gamma _i (\widetilde{\fl} \cdot \widetilde{E}_i ) \\
&\underset{(\ref{en})}{=} \left( \widetilde{\fl} \cdot \psi ^{\ast} \left( K_S - \frac{b}{a}F + \frac{1}{a} \fl \right) \right) + \gamma _n .
\end{align*}
%%%%%%%%%%%%%%%%%%%
Since $K_S - \frac{b}{a}F + \frac{1}{a} \fl \sim _{\bQ} 0$, we have $\gamma _n \le -2$. 
This means $(S,-\frac{b}{a}F+\frac{1}{a}\fl )$ is not log canonical at $x$. 
%%%%%%%%%%%%%%%%%%%
We will consider whether $F$ is smooth or not in what follows. 
%%%%%%%%%%%%%%%%%%%

{\it In the case that $F$ is smooth:} 
By the variant of Corti's inequality, we have: 
%%%%%%%%%%%%%%%%%%%
\begin{align} \label{Corti}
i(L_1,L_2;x) > 4\left( 1+ \frac{b}{a} \right) a^2 = 4a(a+b), 
\end{align}
%%%%%%%%%%%%%%%%%%%
where $L_1$ and $L_2$ are general members of $\fl$. 
Meanwhile since $L_1$ and $L_2$ meet at only $p$, the left hand side of (\ref{Corti}) can be written as: 
%%%%%%%%%%%%%%%%%%%
\begin{align*}
i(L_1,L_2;x) &= \left( \fl ^2 \right) = da(a+4b) \le 4a(a+b),
\end{align*}
%%%%%%%%%%%%%%%%%%%
where we recall that $d$ is less than or equal to $4$. It is a contradiction to (\ref{Corti}). 
%%%%%%%%%%%%%%%%%%%

{\it In the case that $F$ is not smooth:} 
Since $\pi :S \to \bP ^1_k$ is a Mori conic bundle, note that the base extension $F_{\kc}$ of $F$ is the union $E_1 + E_2$ of $(-1)$-curves $E_1$ and $E_2$ on $S_{\kc}$ meeting transversally at $x$ in such a way that $E_1$ and $E_2$ lie in the same $\Gal$-orbit. 
By the variant of Corti's inequality, we have: 
%%%%%%%%%%%%%%%%%%%
\begin{align} \label{Corti(2)}
i(L_1,L_2;x) > 4\left( 1+ \frac{b}{a} + 1+ \frac{b}{a} -1 \right)^2 a^2 = 4a(a+2b), 
\end{align}
%%%%%%%%%%%%%%%%%%%
where $L_1$ and $L_2$ are general members of $\fl$. 
By the similar argument as above, we see: 
%%%%%%%%%%%%%%%%%%%
\begin{align*}
i(L_1,L_2;x) \le 4a(a+b) \le 4a(a+2b),
\end{align*}
%%%%%%%%%%%%%%%%%%%
which is a contradiction to (\ref{Corti(2)}). 
\end{proof}
%%%%%%%%%%%%%%%%%%%
Suppose on the contrary that $S$ contains an $\bA^1_k$-cylinder, say $U \simeq \bA ^1_k \times Z$, where $Z$ is a smooth affine curve defined over $k$. 
The closures in $S$ of fibers of the projection $pr_{Z}: U \simeq Z \times \bA_k^1 \to Z$ yields a linear system, say $\fl$, on $S$. 
%%%%%%%%%%%%%%%%%%%
\begin{claim}\label{basepoint}
The base locus $\Bs (\fl )$ consists of only one $k$-rational point.
\end{claim}
%%%%%%%%%%%%%%%%%%%
\begin{proof} 
Let $\Phi _{\fl}: S \dashrightarrow \overline{Z}$ be the rational map associated with $\fl$, where $\overline{Z}$ is the smooth projective model of $Z$. 
Assume that $\Bs (\fl )$ is base point-free. 
Then $\Phi _{\fl}$ is a morphism, in particular, it is a Mori conic bundle, which admits a section defined over $k$ and is contained in $S \backslash U$, by Lemma \ref{MCB}. 
However, this is a contradiction to Lemma \ref{bdl}(3). 
Thus, $\Bs (\fl )$ is not base point-free. 
By the similar argument as Lemma \ref{lem4-1(2)}, we see that $\Bs (\fl )$ consists of only one $k$-rational point. 
\end{proof} 
%%%%%%%%%%%%%%%%%%%
Let us denote by $p$ the base point of the linear system $\fl$. 
Recall that $S$ is endowed with a structure of a Mori conic bundle $\pi: S \to B$ over a geometrically rational curve $B$ defined over $k$. 
Since $p$ is $k$-rational by Claim \ref{basepoint}, so is its image via $\pi$, in particular, $B$ is isomorphic to $\bP_k^1$. 
Since $Z$ is contained in a projective line $\bP ^1_k$ on $k$ by the similar argument, $\fl$ is a linear pencil on $S$. 
Moreover, we can easily to see $\Pic (S)_{\bQ} = \bQ [-K_S] \oplus \bQ [F]$, where $F$ is a general fiber of $\pi$, which passes through $p$. 
In particular, $\fl$ is $\bQ$-linearly equivalent to $a(-K_S)+bF$ for some rational numbers $a$ and $b$. 
%%%%%%%%%%%%%%%%%%%
\begin{proof}[Proof of Proposition \ref{prop}(3)]
In this proof, we will consider whether $-K_S$ is ample or not as follows. 

At first, we shall consider the case that $-K_S$ is not ample. 
By the assumption, there exists a $\Gal$-orbit of a $(-2)$-curve on $S$, say $M$. 
Then we have $(M \cdot -K_S)=0$, moreover, we can easily to see $(M \cdot F)>0$. 
Thus, we have $b \ge 0$ by virtue of $0 \le (M\cdot \fl )=b(M \cdot F)$. 
However, it is a contradiction to Lemma \ref{Corti lem}. 

Next, we shall consider the case that $-K_S$ is ample. 
By Lemma \ref{Corti lem}, we obtain $a>0$ and $b<0$. 
By Lemma \ref{MCBs}, there exists a Mori conic bundle $\pi _2:S \to \bP ^1_k$ such that a fiber $F_2$ of $\pi _2$ passing through $p$ is linearly equivalent to $\frac{4}{d}(-K_S)-F$. 
Thus, we can write $\fl \sim _{\bQ} (a+\frac{4}{d}b)(-K_S)-bF_2$ with $a+\frac{4}{d}b>0$ and $-b>0$. 
However, it is a contradiction to Lemma \ref{Corti lem}. 

Therefore, $S$ never contains an $\bA _k^1$-cylinder for both cases. 
\end{proof}
%%%%%%%%%%%%%%%%%%%%%%%%%%%%%%%%%%%%%%%%%%%%%%%%%%%%%%%%%%%%%%%%%%%%%%%%%%%%%%%%%%%%%%%%%%%%%%%%%%%%%%%%%%%
\appendix
%%%%%%%%%%%%%%%%%%%%%%%%%%%%%%%%%%%%%%%%%%%%%%%%%%%%%%%%%%%%%%%%%%%
\section{Minimal del Pezzo surfaces of Picard rank two}\label{5}
%%%%%%%%%%%%%%%%%%%
Let $k$ be a field of characteristic zero, and let $S$ be a smooth minimal del Pezzo surface of degree $d$ and of Picard rank $\rho (S)=2$ defined over $k$. 
In this appendix, we give the proof, which says that $d$ is equal to $1,2,4$ or $8$. 
In other words, we shall show $d \not=3,5,6,7,9$. 

We can clearly see $d \not= 7,9$. 
Moreover, {\cite[Theorem 28.1]{M86}} gives the proof for the fact $d\not= 3$. 
Hence, let us prove $d \not= 5,6$: 
%%%%%%%%%%%%%%%%%%%
\begin{proof}[Proof of $d\not= 5,6$]
By Proposition \ref{conic bdl},  $S$ is endowed with a structure of Mori conic bundle defined over $k$, say $\pi : S \to B$. 
Any $(-1)$-curve on $S_{\kc}$, which is not an irreducible component of any singular fiber of $\pi _{\kc}$, meets all singular fibers of $\pi _{\kc}$. 
Notice that $S_{\kc}$ contains exactly $(8-d)$-times of singular fibers of $\pi _{\kc}$, which are the union $E_1+E_2$ of $(-1)$-curves $E_1$ and $E_2$ on $S_{\kc}$ meeting transversely at a point in such a way that $E_1$ and $E_2$ lie in the same $\Gal$-orbit. 
On the other hand, it can be easily seen that any $(-1)$-curve on $S_{\kc}$ meets transversely exactly $(8-d)$-times of $(-1)$-curves on $S_{\kc}$ since $d \ge 5$ and there exists a birational morphism to $\bP ^2_{\kc}$, which is a composite of $(9-d)$-times blow-up. 
Thus, the union of all $(-1)$-curves on $S_{\kc}$, none of which is an irreducible component of any singular fiber of $\pi _{\kc}$, is defined over $k$ and is disjoint. 
It is a contradiction to the minimality of $S$. 
\end{proof}
%%%%%%%%%%%%%%%%%%%%%%%%%%%%%%%%%%%%%%%%%%%%%%%%%%%%%%%%%%%%%%%%%%%
\section{List of all types of weak del Pezzo surfaces}\label{6}
%%%%%%%%%%%%%%%%%%%
This appendix will summarize a classification of types of weak del Pezzo surfaces with anti-canonical divisor not ample in Table \ref{type} according to the triplet $(\text{degree}\ d,\ \text{Singularities},\ \# \,\text{Lines})$ (see also \S \S \ref{2-3}, for the notation). 
\cite{D14, D12, D34} will yield all information for cases of degree $\ge 3$, whereas for cases of degree $\le 2$ few data about the number of lines on such weak del Pezzo surfaces are included there. In order to count the number of lines on them, we need to do somehow a tedious but straightforward calculation by making use of the argument in \ref{3-2-1}. 
As for how to determine $\#$\,Lines for concrete examples, see Example \ref{inB} (for other types, we can check $\#$\,Lines by the similar way). 
%%%%%%%%%%%%%%%%%%%
\begin{eg}\label{inB}
Let $S$ be a weak del Pezzo surface of degree $1$ and $A_3+4A_1$-type and let us determine $\#$\,Lines of $S$. 
We may assume that the base field $k$ of $S$ is an algebraically closed field of characteristic zero. 
We shall take a composite of blowing-ups $(\ref{sigma})$ in such a way that seven $(-2)$-curves on $S_{\kc}$ correspond to $m^3_1$, $m_{1,2}^0$, $m_{2,3}^0$, $m_{4,5}^0$, $m_{6,7}^0$, $m_{4,5,8}^1$ and $m_{6,7,8}^1$ in $I_1$ (for these notation, see (\ref{(-2)})). 
Let $e \in I_1$ be one of those in the list of (\ref{(-1)}) satisfying $(e \cdot m_1^3)\ge 0$, $(e \cdot m_{1,2}^0)\ge 0$, $(e \cdot m_{2,3}^0)\ge 0$, $(e \cdot m_{4,5}^0)\ge 0$, $(e \cdot m_{6,7}^0)\ge 0$, $(e \cdot m_{4,5,8}^1)\ge 0$ and $(e \cdot m_{6,7,8}^1)\ge 0$. 
By Lemma \ref{(-1)-curve} and the argument in \ref{3-2-1}, we note that all elements in $I_1$ such the $e$ above have one-to-one correspondence to all $(-1)$-curves on $S_{\kc}$. 
Hence, we shall determine $(e \cdot m) \ge 0$ for any $m \in I_1$ corresponding to $(-2)$-curve on $S_{\kc}$ by explicit computation. 

For example, supposing $e=-k_1+e_0-(e_{i_1}+e_{i_2}+e_{i_3})$ for $1 \le i_1 < i_2 < i_3 \le 8$, then we see $i_1\not=1$.  Otherwise, we have: 
\begin{align*}
0 \le (e \cdot m_1^3)=3(e_0^2)+2(e_1^2)+(e_{i_2}^2)+(e_{i_3}^2) = 3-2-1-1= -1, 
\end{align*}
which is absurd. 
In particular, we have $(e \cdot e_1)=1$. 
Since $(e \cdot m_{1,2}^0) \ge 0$ and $(e \cdot m_{2,3}^0) \ge 0$, we obtain $(e \cdot e_2)=(e \cdot e_3)=1$. 
Hence, we see $i_1 \not= 2,3$, namely, $4 \le i_1<i_2<i_3 \le 8$. 
However, we have: 
\begin{align*}
0 \le (e \cdot m_{4,5,8}^1+m_{6,7,8}^1) = 2 + (e_{i_1}+e_{i_2}+e_{i_3} \cdot e_4+e_5+e_6+e_7+2e_8) \le -1. 
\end{align*}
This is a contradiction. 
Thus, we obtain $e \not= -k_1+e_0-(e_{i_1}+e_{i_2}+e_{i_3})$ for any $1 \le i_1 < i_2 < i_3 \le 8$. 

By a similar argument, we can check that $e \not= -k_1+2e_0-(e_{i_1}+\dots +e_{i_6})$ for any $1 \le i_1 < \dots < i_6 \le 8$ nor $e \not= -2k_1-e_i$ for any $1 \le i \le 8$. 
Furthermore, we can also check the following: 
\begin{itemize}
\item If $e=e_i$ $(1 \le i \le 8)$, then we obtain $i=3,\,4,\,6,\,8$. 
\item If $e=\ell _{i,j}$ $(1 \le i < j \le 8)$, then we obtain $(i,j)=(1,2)$, $(1,4)$, $(1,6)$, $(1,8)$, $(4,6)$. 
\item If $e=-k_1+(e_{i_1}+e_{i_2}+e_{i_3})$ $(1 \le i_1 < i_2 < i_3 \le 8)$, then we obtain $(i_1,i_2,i_3)=(2,3,8)$, $(3,5,7)$, $(3,5,8)$, $(3,7,8)$, $(4,5,7)$, $(4,5,8)$, $(5,6,7)$, $(5,7,8)$, $(6,7,8)$. 
\item If $e=-k_1-e_1+e_j$ $(1 \le i,\, j \le 8,\ i \not= j)$, then we obtain $(i,j)=(4,5)$, $(4,8)$, $(6,7)$, $(6,8)$. 
\end{itemize}
Therefore, we have $\#$\,Lines $=4+5+9+4=22$ (cf. Table \ref{type}). 
\end{eg}
%%%%%%%%%%%%%%%%%%%
\begin{center}
 \begin{longtable}{|c|c||c|c||c|c|}
\caption{The list of the types of weak del Pezzo surfaces. }\label{type} \\
 \cline{1-4}
\multicolumn{2}{|c||}{Degree $8$} & \multicolumn{2}{c||}{Degree $7$} & \multicolumn{2}{c}{} \\ \cline{1-4}
\endfirsthead
\caption{Continued. } \\
 \hline
%Singularities & $\#$\,Lines & Singularities & $\#$\,Lines & Singularities & $\#$\,Lines \\ \hline
\endhead

Singularities & $\#$\,Lines & Singularities & $\#$\,Lines & \multicolumn{2}{c}{} \\ \cline{1-4}
$A_1$ & $0$ & $A_1$ & $2$ & \multicolumn{2}{c}{} \\ \hline \hline

\multicolumn{1}{|c}{\qquad \qquad \qquad \qquad} & \multicolumn{1}{c}{} & \multicolumn{2}{c}{Degree $6$} & \multicolumn{1}{c}{\qquad \qquad \qquad \qquad} & \multicolumn{1}{c|}{} \\ \hline 
Singularities & $\#$\,Lines & Singularities & $\#$\,Lines & Singularities & $\#$\,Lines \\ \hline
$A_2+A_1$ & $1$ & $A_2$ & $2$ & $2A_1$ & $2$ \\ \hline
$(A_1)_<$ & $3$ & $(A_1)_>$ & $4$ \\ \hline \hline

\multicolumn{6}{|c|}{Degree $5$} \\ \hline
Singularities & $\#$\,Lines & Singularities & $\#$\,Lines & Singularities & $\#$\,Lines \\ \hline
$A_4$ & $1$ & $A_3$ & $2$ & $A_2+A_1$ & $3$ \\ \hline
$A_2$ & $4$ & $2A_1$ & $5$ & $A_1$ & $7$ \\ \hline \hline

\multicolumn{6}{|c|}{Degree $4$} \\ \hline
Singularities & $\#$\,Lines & Singularities & $\#$\,Lines & Singularities & $\#$\,Lines \\ \hline
$D_5$ & $1$ & $A_3+2A_1$ & $2$ & $D_4$ & $2$ \\ \hline
$A_4$ & $3$ & $A_3+A_1$ & $3$ & $A_2+2A_1$ & $4$ \\ \hline
$4A_1$ & $4$ & $(A_3)_<$ & $4$ & $(A_3)_>$ & $5$ \\ \hline
$A_2+A_1$ & $6$ & $3A_1$ & $6$ & $A_2$ & $8$ \\ \hline
$(2A_1)_<$ & $8$ & $(2A_1)_>$ & $9$ & $A_1$ & $12$ \\ \hline \hline

\multicolumn{6}{|c|}{Degree $3$} \\ \hline
Singularities & $\#$\,Lines & Singularities & $\#$\,Lines & Singularities & $\#$\,Lines \\ \hline
$E_6$ & $1$ & $A_5+A_1$ & $2$ & $3A_2$ & $3$ \\ \hline
$D_5$ & $3$ & $A_5$ & $3$ & $A_4+A_1$ & $4$ \\ \hline
$A_3+2A_1$ & $5$ & $2A_2+A_1$ & $5$ & $D_4$ & $6$ \\ \hline
$A_4$ & $6$ & $A_3+A_1$ & $7$ & $2A_2$ & $7$  \\ \hline
$A_2+2A_1$ & $8$ & $4A_1$ & $9$ & $A_3$ & $10$ \\ \hline
$A_2+A_1$ & $11$ & $3A_1$ & $12$ & $A_2$ & $15$ \\ \hline
$2A_1$ & $16$ & $A_1$ & $21$ \\ \hline \hline

\multicolumn{6}{|c|}{Degree $2$} \\ \hline
Singularities & $\#$\,Lines & Singularities & $\#$\,Lines & Singularities & $\#$\,Lines \\ \hline
$E_7$ & $1$ & $A_7$ & $2$ & $D_6+A_1$ & $2$ \\ \hline
$A_5+A_2$ & $3$ & $D_4+3A_1$ & $4$ & $2A_3+A_1$ & $4$ \\ \hline
$E_6$ & $4$ & $D_6$ & $3$ & $A_6$ & $4$ \\ \hline
$D_5+A_1$ & $5$ & $(A_5+A_1)_<$ & $5$ & $(A_5+A_1)_>$ & $6$ \\ \hline
$D_4+2A_1$ & $6$ & $A_4+A_2$ & $6$ & $2A_3$ & $6$ \\ \hline
$A_3+A_2+A_1$ & $7$ & $A_3+3A_1$ & $8$ & $3A_2$ & $8$ \\ \hline
$6A_1$ & $10$ & $D_5$ & $8$ & $(A_5)_<$ & $7$ \\ \hline
$(A_5)_>$ & $8$ & $D_4+A_1$ & $9$ & $A_4+A_1$ & $10$ \\ \hline
$A_3+A_2$ & $10$ & $(A_3+2A_1)_<$ & $11$ & $(A_3+2A_1)_>$ & $12$ \\ \hline
$2A_2+A_1$ & $12$ & $A_2+3A_1$ & $13$ & $5A_1$ & $14$ \\ \hline
$D_4$ & $14$ & $A_4$ & $14$ & $(A_3+A_1)_<$ & $15$  \\ \hline
$(A_3+A_1)_>$ & $16$ & $2A_2$ & $16$ & $A_2+2A_1$ & $18$ \\ \hline
$(4A_1)_<$ & $19$ & $(4A_1)_>$ & $20$ & $A_3$ & $22$ \\ \hline
$A_2+A_1$ & $24$ & $(3A_1)_<$ & $25$ & $(3A_1)_>$ & $26$ \\ \hline
$A_2$ & $32$ & $2A_1$ & $34$ & $A_1$ & $44$ \\ \hline \hline

\multicolumn{6}{|c|}{Degree $1$} \\ \hline
Singularities & $\#$\,Lines & Singularities & $\#$\,Lines & Singularities & $\#$\,Lines \\ \hline
$E_8$ & $1$ & $D_8$ & $2$ & $A_8$ & $3$ \\ \hline
$E_7+A_1$ & $3$ & $A_7+A_1$ & $5$ & $E_6+A_2$ & $4$ \\ \hline
$D_6+2A_1$ & $5$ & $D_5+A_3$ & $5$ & $A_5+A_2+A_1$ & $8$ \\ \hline
$2D_4$ & $5$ & $2A_4$ & $6$ & $2A_3+2A_1$ & $11$ \\ \hline
$4A_2$ & $12$ & $E_7$ & $5$ & $D_7$ & $5$ \\ \hline
$(A_7)_<$ & $7$ & $(A_7)_>$ & $8$ & $E_6+A_1$ & $8$ \\ \hline
$D_6+A_1$ & $9$ & $A_6+A_1$ & $10$ & $D_5+A_2$ & $10$ \\ \hline
$D_5+2A_1$ & $12$ & $A_5+A_2$ & $12$ & $A_5+2A_1$ & $14$ \\ \hline
$D_4+A_3$ & $11$ & $D_4+3A_1$ & $17$ & $A_4+A_3$ & $12$ \\ \hline
$A_4+A_2+A_1$ & $15$ & $2A_3+A_1$ & $16$ & $A_3+A_2+2A_1$ & $19$ \\ \hline
$A_3+4A_1$ & $22$ & $3A_2+A_1$ & $20$ & $E_6$ & $13$ \\ \hline
$D_6$ & $13$ & $A_6$ & $15$ & $D_5+A_1$ & $18$ \\ \hline
$(A_5+A_1)_<$ & $20$ & $(A_5+A_1)_>$ & $21$ & $D_4+A_2$ & $20$ \\ \hline
$D_4+2A_1$ & $24$ & $A_4+A_2$ & $22$ & $A_4+2A_1$ & $25$ \\ \hline
$(2A_3)_<$ & $22$ & $(2A_3)_>$ & $23$ & $A_3+A_2+A_1$ & $27$ \\ \hline
$A_3+3A_1$ & $31$ & $3A_2$ & $29$ & $2A_2+2A_1$ & $32$ \\ \hline
$A_2+4A_1$ & $36$ & $6A_1$ & $41$ & $D_5$ & $27$ \\ \hline
$A_5$ & $29$ & $D_4+A_1$ & $34$ & $A_4+A_1$ & $36$ \\ \hline
$A_3+A_2$ & $38$ & $(A_3+2A_1)_<$ & $43$ & $(A_3+2A_1)_>$ & $44$ \\ \hline
$2A_2+A_1$ & $45$ & $A_2+3A_1$ & $50$ & $5A_1$ & $56$ \\ \hline
$D_4$ & $49$ & $A_4$ & $51$ & $A_3+A_1$ & $60$  \\ \hline
$2A_2$ & $62$ & $A_2+2A_1$ & $69$ & $(4A_1)_<$ & $76$ \\ \hline
$(4A_1)_>$ & $77$ & $A_3$ & $83$ & $A_2+A_1$ & $94$ \\ \hline
$3A_1$ & $103$ & $A_2$ & $127$ & $2A_1$ & $138$ \\ \hline
$A_1$ & $183$ & \multicolumn{1}{c}{\qquad \qquad \qquad \qquad} & \multicolumn{3}{c}{} \\ \cline{1-2}
\end{longtable}
 \end{center}
%%%%%%%%%%%%%%%%%%%

%%%%%%%%%%%%%%%%%%%%%%%%%%%%%%%%%%%%%%%%%%%%%%%%%%%%%%%%%%%%%%%%%%%%%%%%%%%%%%%%%%%%%%%%%%%%%%%%%%%%%%%%%%%
\subsection*{Acknowledgement}
The author is deeply grateful to his supervisor Professor Takashi Kishimoto for his helpful advice with warm encouragement during the preparation of the article. 
He is further thankful to Professors Ivan Cheltsov and Kiwamu Watanabe for giving some useful advice. 
Also, he would like to thank the referees for suggesting many valuable comments that helped to improve this article. 
%for their useful advices about the usage of Corti's inequality and root systems. 

This work was supported by the JGC-S (Nikki-Saneyoshi) Scholarship Foundation. 
%%%%%%%%%%%%%%%%%%%%%%%%%%%%%%%%%%%%%%%%%%%%%%%%%%%%%%%%%%%%%%%%%%%%%%%%%%%%%%%%%%%%%%%%%%%%%%%%%%%%%%%%%%%

\end{document}